\documentclass[openany, oneside]{book}
\usepackage{amsfonts}
\usepackage{setspace}
\usepackage{amsthm}
\usepackage{amsmath}
\theoremstyle{plain}
\def\upint{\mathchoice%
    {\mkern13mu\overline{\vphantom{\intop}\mkern7mu}\mkern-20mu}%
    {\mkern7mu\overline{\vphantom{\intop}\mkern7mu}\mkern-14mu}%
    {\mkern7mu\overline{\vphantom{\intop}\mkern7mu}\mkern-14mu}%
    {\mkern7mu\overline{\vphantom{\intop}\mkern7mu}\mkern-14mu}%
  \int}
\def\lowint{\mkern3mu\underline{\vphantom{\intop}\mkern7mu}\mkern-10mu\int}

\newtheorem{theorem}{Theorem}
\newtheorem{lemma}{Lemma}
\newtheorem{definition}{Definition}
\newtheorem{notation}{Notation}

\begin{document} 
\mainmatter

\begin{center}

{\Large
\textbf{A Sequential Approach to the Henstock Integral} \\
}
{\large
\bigskip
Laramie Paxton  \\

Washington State University\\
}
realtimemath@gmail.com

\end{center}
\bigskip

{\large
\noindent \textbf{Abstract:} The theory of integration over $\mathbb{R}$ is rich with techniques as well as necessary and sufficient conditions under which integration can be performed. Of the many different types of integrals that have been developed since the days of Newton and Leibniz, one relative newcomer is that of the Henstock integral, aka the Henstock-Kurzweil  integral, Generalized Riemann integral, or gauge integral, which was discovered independently by Henstock and Kurzweil in the mid-1950s. In this paper, we develop an alternative, sequential definition of the Henstock integral over closed intervals in $\mathbb{R}$ that we denote as the Sequential Henstock integral. We show its equivalence to the standard  $\epsilon - \delta$ definition of the Henstock integral as well as to the Darboux definition and to a topological definition of the Henstock integral. We then establish the basic properties and fundamental theorems, including two convergence theorems, for the Sequential Henstock integral and offer several suggestions for further study.

}

\bigskip

\bigskip

{\large
\noindent \textbf{Acknowledgements:} We would like to acknowledge Dr. Peng Yee Lee for his support and excellent ideas regarding this project. Without his suggestion to pursue the Sequential Henstock integral, this work would not have developed.
}

\section*{Introduction} 
{\large
The Riemann integral, which is familiar to the reader from calculus, can be loosely defined as the limit of the sum of narrow ``strips" of area under a function taken as the width of each strip goes to 0. While this is the standard integral used in introductory calculus courses, the limitations of the Riemann integral are well known \cite{Abb01}, \cite{Bar01}, \cite{Gor94}, \cite{LV00}, \cite{Tho13}, \cite{Pfe93}, which we discuss below. Thus, since the early 20th century, mathematicians have sought to refine the rigorous foundation of integration as well as enlarge the class of integrable functions in search of further applications of integration theory. Researchers like Lebesgue, Denjoy, Perron, Henstock, and Kurzweil all made significant achievements towards this aim, with the Henstock integral being one of the more recent of these, having been discovered independently by English mathematician Ralph Henstock and Czech mathematician Jaroslav Kurzweil in 1955 and 1957, respectively~\cite{Bar96}~\cite{Sul11}. 

The problem that we examine in this paper is that of developing the theory for a specific definition of the Henstock integral that was defined, but not developed, by Lee \cite{Lee08}, called the Sequential Henstock integral. Starting with Lee's definition, we show the equivalence of the Sequential Henstock integral to several other definitions of the Henstock integral, and we prove the basic properties of and several fundamental theorems for the Sequential Henstock integral. 

As a motivating example, consider a simple function, such as $f(x)=-x^2+4x$, which is an upside down parabola crossing the $x$-axis at $(0,0)$ and $(0,4)$. At its most basic level, we can view integration as finding the area underneath a curve. So, in this case, we want to find the area enclosed between the parabola and the $x$-axis. The different types of integrals that have been developed simply allow us to take various approaches to finding this area. As noted above, the Riemann approach is to take vertical strips of infintesimal width and add up the area contribution of each width. Lebesgue's idea was to use horizontal strips, based on the range values of the function, in order to integrate functions that, loosely speaking, have too many jumps and gaps. The main idea behind the Henstock integral is very similar to the Riemann integral, but we can now use different widths at different points of the function for our vertical strips in order to accommodate jumps and gaps. Returning to our example, we find that the area in consideration equals $32/3$, and the Riemann integral works quite nicely for such a well behaved function. Not all functions are smooth and continuous like our parabolic function though, hence the motivation behind finding other approaches to integration. 

The relevance of developing the theory for the Sequential Henstock integral is twofold. First, as discussed below, there are limitations to the most commonly used types of integrals and the types of functions they are able to integrate, the Riemann integral being one of them. These limitations led to the development of the Henstock integral, as mentioned above, but in its current form, researchers have met with limited success at extending the Henstock integral beyond $\mathbb{R}$, the set of real numbers \cite{Sch09}.

Therefore, the motivation for this research mostly comes from the potential for expanding the overall theory of Henstock integration into more abstract mathematical settings, which in turn may lead to further applications of the Henstock integral. This is described in more detail in the suggestions for further study at the end of the paper, but as an example, by using the Sequential Henstock integral and the theory herein developed, we think that it will be possible to extend the theory of Henstock integration more readily to abstract spaces through the use of generalized sequences, which the original Henstock integral does not use. 

Another important reason behind this research is that the use of the Sequential Henstock integral has been suggested as a way to better facilitate beginning calculus students' understanding of integration \cite{Lee08}, so with the development of the theory behind the Sequential Henstock integral, it will now be possible to test this hypothesis.

\section*{Glossary of Useful Definitions, Lemmas \& Theorems}
{\large

\begin{center}
\textbf{Useful Definitions \& Notation}
\end{center}


\noindent \textbf{Compact:} A set $K \subseteq \mathbb{R}$ is \emph{compact} if every sequence in $K$ has a subsequence that converges to a limit that is also in $K$. Equivalently,  $K \subseteq \mathbb{R}$ is compact if and only if it is both closed and bounded  \cite{Abb01}. 
\newline

\begin{notation}\label{N:Seq.Def}

Throughout this paper, let $I$ denote a
compact interval $[a, b] \subset \mathbb{R}$ and let \\ $i, j, k, m, n, p\in \mathbb{N}$. We denote the class of \emph{Sequential Henstock integrable} functions on $I$ as $H^*(I)$.
\end{notation}
\bigskip


\noindent \textbf{Absolutely Integrable:} A function $f:I\to \mathbb{R}$ is said to be \emph{absolutely integrable} on $I$ if $|f|$ is also integrable on $I$ \cite{Bar01}.
\newline


\noindent \textbf{Cauchy Sequence:} A sequence $\{a_n\}$ is a \emph{Cauchy sequence} if for every $\epsilon>0$, there exists an $N \in \mathbb{N}$ such that if $m,n>N$, we have $|a_m-a_n|<\epsilon$ \cite{Abb01}. 
\newline


\noindent \textbf{Metric Space:} A \emph{metric space} is a set $X$ together with a metric $d$, where $d$ is a measure of ``distance" that exhibits the properties of reflexivity, symmetry, and transitivity \cite{Abb01}. 
\newline


\noindent \textbf{Complete (Metric Space):} A metric space $(X,d)$ is \emph{complete} if every Cauchy sequence in $X$ converges to an element of $X$ \cite{Abb01}. 
\newline


\noindent \textbf{Completeness of $\mathbb{R}$:} The \emph{completeness} of $\mathbb{R}$ is defined such that every nonempty set of real numbers that is bounded above has a least upper bound \cite{Abb01}. 
\newline


\noindent \textbf{$\delta$-fine:} Let $\delta>0$. A partition $p$ is \emph{$\delta$-fine} if every subinterval 
$[x_{i-1}, x_i]$ satisfies $x_i-x_{i-1}<\delta$ \cite{Bar01}.



\noindent \textbf{$\delta(x)$-fine:} Let $\delta(x)>0$. A tagged partition $P=\{([x_{i-1}, x_i], t_i)\}_{i=1}^k$ is \emph{$\delta(x)$-fine} if every subinterval 
$[x_{i-1}, x_i]$ satisfies $x_i-x_{i-1}<\delta(t_i)$ \cite{Abb01}.
\newline


\noindent \textbf{Gauge:} A \emph{gauge} is a function $\delta:I \to \mathbb{R}$ such that 
$\delta(x)>0$ for each $x \in I$ \cite{Bar01}.
\newline


\noindent \textbf{Hausdorff Space:} A space $X$ is called a \emph{Hausdorff space} if for each pair $x_1, x_2$ of distinct points of $X$, there exist neighborhoods $U_1$ and $U_2$ of $x_1$ and $x_2$, respectively, that are disjoint \cite{Mun00}.
\newline


\noindent \textbf{Henstock Integral ($\epsilon-\delta$ Definition):} A function $f:I \to \mathbb{R}$ is \emph{Henstock integrable} on $I$ if there exists a number $A$ such that for each $\epsilon>0$ there exists a $\delta(x)>0$ such that for all $\delta(x)$-fine tagged partitions $P$ of $I$, we have $|S(f,P)-A|<\epsilon$. We say that $A$ is the Henstock integral of $f$ on $I$ with $A=\int_If$ \cite{Bar01}.
\newline


\noindent \textbf{Improper Riemann Integral:} An \emph{improper Riemann integral} is the limit of a definite Riemann integral taken as one or both of the endpoints of the interval tend to a given real number or to $\pm\infty$. For example,
\[
\int_a^\infty f \quad \Longleftrightarrow \quad \lim_{b\to\infty} \int_a^b f \qquad \cite{Rud76}.
\]


\noindent \textbf{Integrable:} A function $f:I \to \mathbb{R}$ is \emph{integrable} under a given type of integral (e.g. Riemann, Lebesgue, or Henstock) if the value of the integral of $f$ is defined for the given type of integral, i.e. contained in the class of functions that can be integrated using the integral. 

\emph{Note: The Riemann integral of a function may not be defined, but the Lebesgue and/or Henstock integral may be (see (\ref{dirichlet}), Dirichlet's function, for an example of such a function).}
\newline


\noindent \textbf{Lebesgue Integral:} \emph{Note: The formal definition of the Lebesgue integral is outside of the scope of this paper, requiring many preliminary notions, so we shall instead define the class of functions that can be integrated with it.} 

A function $f:I \to \mathbb{R}$ is \emph{Lebesgue integrable} if both $f$ and $|f|$ are Henstock integrable or Sequential Henstock integrable. In other words, $f$ is absolutely integrable \cite{Bar01}.
\newline


\noindent \textbf{Locally Compact:} A space $X$ is said to be \emph{locally compact}  if for each point $x\in X$, there exists a compact subspace $C_x$ that contains a neighborhood of $x$ \cite{Mun00}.
\newline


\noindent \textbf{Measure Zero:} A set $Z \subset \mathbb{R}$ is said to be a   \emph{set of measure zero} if for every $\epsilon>0$ there exists a countable collection $\{J_k\}^\infty_{k=1}$ of open intervals such that
\[
Z \subseteq \bigcup_{k=1}^\infty J_k \qquad \text{and} \qquad \sum_{k=1}^\infty l(J_k) \leq \epsilon \qquad \cite{Bar01}.
\]
\emph{See Notation \ref{N:RieSum1} below regarding the use of $l(J_k)$.}
\newline

\noindent \textbf{Monotone Function:} A function $f:I\to \mathbb{R}$ is said to be \emph{monotone} if it is either non-decreasing or non-increasing on  $ \mathbb{R}$.
\newline


\noindent \textbf{Partition:} A \emph{partition} $p$ of an interval $I$ is a finite set of points $\{x_i\}_{i=0}^k$ such that 

\begin{center} $a=x_0<x_1< \dots <x_{k-1}<x_k=b$ \qquad \cite{Bar01}. \end{center}



\noindent \textbf{Riemann Sum:} Let $f:I \to \mathbb{R}$ and let $P=\{([x_{i-1}, x_i], t_i)\}_{i=1}^k$ be a tagged partition of $I$. Then the \emph{Riemann sum} of $f$ on $P$ is defined by
\begin{equation}\label{E:Riesum1}
S(f,P)=\sum_{i=1}^k f(t_i)(x_i-x_{i-1}) \qquad \cite{Bar01}.
\end{equation}

\begin{notation}\label{N:RieSum1}
We will often write the right hand side of (\ref{E:Riesum1}) as $\sum_{i=1}^k f(t_i)l(I_i)$, where~$l(I_i)=(x_i-x_{i-1})$ for $i=1,2, \dots,k$.

\end{notation}

\noindent \textbf{Sequential Henstock Integral:}  A function $f:I \to \mathbb{R}$ is \emph{Sequential Henstock integrable} on $I$ if there exists a number $A$ and a sequence of positive functions $\{\delta_n(x)\}_{n=1}^\infty$ such that
for every $\delta_n(x)$-fine tagged partition $P_n$, we have

\begin{equation}\label{E:Seq}
S(f,P_n)=\sum_{i=1}^{m_n \in \mathbb{N}} f(t_{i_n})(x_{i_n}-x_{{(i-1)}_n}) \longrightarrow A \text{ as } n \longrightarrow \infty.
\end{equation}

\noindent We say that $A$ is the Sequential Henstock integral of $f$ on $I$ with $A=\int_If$ \cite{Lee08}.
\newline


\noindent \textbf{Tagged Partition:} A \emph{tagged partition} $P$ of an interval $I$ is an ordered pair consisting of a partition $p=\{x_i\}_{i=0}^k$ of $I$ together with a set $\{t_i\}_{i=1}^k$ of points
such that $x_{i-1}\leq t_i \leq x_i$ for each $i$, and we write $P=\{([x_{i-1}, x_i], t_i)\}_{i=1}^k$ \cite{Bar01}.
\newline


\noindent \textbf{Total Boundedness:} A metric space $(X,d)$ is said to be \emph{totally bounded} if for every $\epsilon >0$, there exists a finite covering of $X$ by $\epsilon$-balls; i.e. there are a finite number of open sets (balls) of radius $\epsilon$ whose union is equal to $X$ \cite{Mun00}.

}

{\large

\begin{center}
\textbf{Useful Lemmas \& Theorems for Sequential Henstock Integrable Functions}
\end{center}

\noindent Note: While we have adapted many of the following lemmas and theorems for Sequential Henstock integrals, the source(s) noted for each contain(s) the theorems for Lebesgue or Henstock integrals, where appropriate.
\newline

\noindent \textbf{Cauchy Criterion for Sequential Henstock Integrability:} Let $f:I \to \mathbb{R}$. Then $f$ is Sequential Henstock integrable on $I$ if and only if for every $\epsilon>0$, there exists a $\delta_N(x) \in \{\delta_n(x)\}_{n=1}^\infty$ such that for $n\geq N$ and for all $\delta_n(x)$-fine tagged partitions $P_n$ and $Q_n$ of $I$ we have
\[
|S(f,P_n)-S(f,Q_n)|<\epsilon \qquad \cite{Bar01}.
\]



\noindent \textbf{Cousin's Lemma:} For each gauge $\delta(x)$ defined on $I$, there exists a $\delta(x)$-fine tagged partition of $I$ \cite{Bar01}.
\newline

\noindent \textbf{Dominated Convergence Theorem:} Let $\{f_k\}$ be a sequence of Sequential Henstock integrable functions defined on $I=[a,b]$ that converges pointwise to a limit function $f:[a,b]\to \mathbb{R}$. Suppose that there exist functions $\alpha, \omega \in H^*(I)$ such that
\[
\alpha(x) \leq f_k(x) \leq \omega(x) \qquad \text{for } x \in I, k \in \mathbb{N}.
\]
Then $f\in H^*(I)$ and $$\int_If=\lim_{k \to \infty} \int_If_k  \qquad \cite{Bar01}, \cite{Rud76}.$$


\noindent \textbf{Fatou's Lemma:} Let $f_k, \alpha \in H^*(I)$ be such that

\[
\alpha(x) \leq f_k(x) \qquad \text{for } x \in I, k \in \mathbb{N}.
\]
and that 
\[
\liminf_{k\to \infty}\int_If_k<\infty.
\]
\newline
Then $\liminf_{k\to \infty}f_k$ belongs to $H^*(I)$ and

\[
-\infty < \int_I\liminf_{k\to \infty}f_k \leq \liminf_{k\to \infty}\int_If_k < \infty.
\]

\noindent \emph{Note that $\inf$ denotes the infimum of a set and $\liminf$ denotes the limit inferior of a set.} \cite{Bar01},~\cite{Rud76}.


\noindent \textbf{Heine-Borel Theorem:} A set $K \subseteq \mathbb{R}$ is compact if and only if it is closed and bounded \cite{Abb01}. 
\newline

\noindent \textbf{Henstock's Lemma:} Let $f:[a,b]\to \mathbb{R}$ be a Sequential Henstock integrable function. If
$P_n^*=\{([x_{i-1}, x_i], t_i)\}_{i=1}^p$ is any $\delta_n(x)$-fine set of disjoint, tagged subintervals of $[a,b]$, then
\begin{equation}\label{hen1}
\left| S(f,P^*_n)-\sum_{i=1}^p \int_{x_{i-1}}^{x_i}f \right| \leq \epsilon.
\end{equation}

\noindent Further,
\begin{equation}\label{hen2}
\sum_{i=1}^p \left| f(t_i)(x_i-x_{i-1})- \int_{x_{i-1}}^{x_i}f \right| \leq 2\epsilon \qquad \cite{Bar01}.
\end{equation}
\newline


\noindent \textbf{Monotone Convergence Theorem:} Let $\{f_k\}$ be a monotone sequence of Sequential Henstock integrable functions defined on $I=[a,b]$ that converges pointwise to a limit function $f:[a,b]\to \mathbb{R}$. If $\lim_{k \to \infty} \int_If_k$ exists, then $f\in H^*(I)$ and $$\int_If=\lim_{k \to \infty} \int_If_k \qquad \cite{Bar01}, \cite{Rud76}.$$
\newline

\noindent \textbf{Uniform Convergence Theorem:} If $\{f_k\} \in H^*(I)$ converges to $f$ uniformly on $I$, then $f\in H^*(I)$ and 
$$\int_If=\lim_{k \to \infty} \int_If_k \qquad \cite{Bar01}, \cite{Rud76}.$$ 

}

\section*{Common Integrals}
{\large

\begin{center}
\textbf{Riemann Integral}
\end{center}

We now highlight two key limitations of the Riemann integral. Lebesgue (1901) proved that in order for a bounded function to be Riemann integrable, its points of discontinuity must be of measure zero \cite{Abb01}. So, we immediately can find a large class of functions that are not Riemann integrable, such as Dirichlet's function (\ref{dirichlet}), which is discontinuous at every point in $\mathbb{R}$ \cite{Abb01}.

\begin{equation}\label{dirichlet}
f(x)= \left \{ 
\begin{array}{lr} 
1 &x\in \mathbb{Q} \\
0 &x\notin \mathbb{Q} \\
\end{array} 
\right.
\end{equation}
\bigskip

A further complication resulting from the limitations of the Riemann integral involves Part I of the Fundamental Theorem of Calculus: If $f:[a,b]\to \mathbb{R}$ is integrable (in the Riemann sense), and $F:[a,b]\to \mathbb{R}$ satisfies $F'(x)=f(x)$ for all $x\in [a,b]$, then 

\begin{equation}\label{FTC}
\int_a^bf=F(b)-F(a) \qquad \cite{Abb01}.
\end{equation}

Under Riemann integration, the hypothesis requires $f$ to be Riemann integrable or else (\ref{FTC}) does not hold. This limits the class of functions to which we can apply (\ref{FTC})  when using Riemann integration. (See \cite{Abb01}, pp. 207-210, for an example of a differentiable function whose derivative is non-Riemann integrable.) 

\begin{center}
\textbf{Lebesgue Integral}
\end{center}

Lebesgue (1902) introduced the current integral of choice in advanced mathematics, \\ allowing for a much larger class of integrable functions and especially for stronger convergence theorems for sequences of integrable functions, such as the Monotone Convergence Theorem and Dominated Convergence Theorem~\cite{Swa01},~\cite{Tho13}. Nevertheless, the Lebesgue integral does have its limitations \cite{Abb01} (although some would argue that these do not interfere with its intended purpose of ensuring completeness in certain metric function spaces through its powerful convergence theorems \cite{Ada13}, \cite{Rek12}). 

First, in order to be Lebesgue integrable, not only must a function $f$ be Lebesgue integrable, but so must its absolute value $|f|$. In other words, Lebesgue integrable functions must be absolutely integrable \cite{Bar01}. Regarding the Fundamental Theorem of Calculus (\ref{FTC}), there are derivatives that are not Lebesgue integrable, and there are also functions that have an improper Riemann integral defined that are not Lebesgue integrable~\cite{Abb01}. 

One other potentially limiting aspect of Lebesgue's integration theory is that it is, in general, rather complicated and requires significant preparation for students learning the theory. Thus, it is rarely taught at the undergraduate level, and the Riemann integral is still the focus there due to its more simplistic definition and easy-to-visualize geometry \cite{Lee08}, \cite{Mac14}, \cite{Mcs73}.

\begin{center}
\textbf{Denjoy, Henstock \& Other Integrals}
\end{center}

Not long after Lebesgue published his theory of integration in 1904, other integration theorists developed methods to integrate functions that were not absolutely integrable \cite{Sul11}. Denjoy sought to integrate functions of extreme oscillation, such as the following example that can be seen to oscillate more and more rapidly as we approach the origin from the right, tending to infinity as $x$ tends to 0. This integral is neither Riemann integrable nor Lebesgue integrable, yet as we shall see, it is Henstock and Sequential Henstock integrable.

\begin{equation}\label{denjoy}
f(x)=\int_0^1\frac{1}{x}\sin{\frac{1}{x^3}}dx \qquad \cite{Sul11}
\end{equation}

\indent In 1912 he published what became known as the Denjoy integral. Other notable integrals of a similar nature are the Luzin (1912), Perron (1914), and Fenyman (1948) integrals \cite{Sul11}. Henstock, Lee, Muldowney, Gordon, and others have since shown that all of these integrals along with others, such as the McShane, the Ito, the Haar, and the Kunugi integrals, are equivalent to the Henstock integral in that they can integrate the same classes of functions on $\mathbb{R}$ and yield the same values for the integral  \cite{Gor94}, \cite{Hen91}, \cite{Lee89}, \cite{Mul00}, \cite{TC05}. 

In fact, Riemann integrals and Lebesgue integrals have been shown to be special cases of the Henstock integral (over $\mathbb{R}$) \cite{Gor94}, \cite{LV00}, \cite{Mcs73}, \cite{Pfe93}. Thus, while the Henstock integral does not generalize to abstract spaces as readily or completely as the Lebesgue integral \cite{Sch09} (for reasons we will see below), the fact that the Lebesgue integral can be derived from the Henstock integral ensures a unified theory of integration whether over the real line or abstract function spaces \cite{Gor94}, \cite{Mcs73}, \cite{Pfe93}.

}
\section*{The Riemann Integral as a Limit of Sums}

{\large
For the convenience of the reader, we now restate several definitions and an important lemma as shown in \cite{Bar01}, \cite{Swa01}, and \cite{Wel11} 
aimed at presenting the Riemann integral formally
as a limit of Riemann sums using 
an $\epsilon - \delta$ definition, from which the definition of the Henstock integral in the glossary is derived. We then state the Sequential Henstock integral definition along with two additional definitions of the 
Henstock integral and show that the Sequential Henstock 
integral is equivalent to each of the three other definitions stated for the Henstock integral.
\newline

\noindent \textbf{Notation \ref{N:Seq.Def}.}
Throughout this paper, let $I$ denote a
compact interval $[a, b] \subset \mathbb{R}$ and let \\ $i, j, k, m, n, p\in \mathbb{N}$. We denote the class of \emph{Sequential Henstock integrable} functions on $I$ as $H^*(I)$.

\begin{definition}\label{D:partition}

A \emph{partition} $p$ of an interval $I$ is a finite set of points $\{x_i\}_{i=0}^k$ such that 

\begin{center} $a=x_0<x_1< \dots <x_{k-1}<x_k=b$. \end{center}

\end{definition}

\begin{definition}\label{D:taggedpartition}

A \emph{tagged partition} $P$ of an interval $I$ is an ordered pair consisting of a partition $p=\{x_i\}_{i=0}^k$ of $I$ together with a set $\{t_i\}_{i=1}^k$ of points
such that $x_{i-1}\leq t_i \leq x_i$ for each $i$, and we write $P=\{([x_{i-1}, x_i], t_i)\}_{i=1}^k$.

\end{definition}

\begin{definition}\label{D:RiemannSum}

Let $f:I \to \mathbb{R}$ and let $P=\{([x_{i-1}, x_i], t_i)\}_{i=1}^k$ be a tagged partition of $I$. Then the \emph{Riemann sum} of $f$ on $P$ is defined by
\begin{equation}\label{E:Riesum}
S(f,P)=\sum_{i=1}^k f(t_i)(x_i-x_{i-1}).
\end{equation}
\end{definition}

\noindent \textbf{Notation \ref{N:RieSum1}.}
We will often write the right hand side of (\ref{E:Riesum}) as $\sum_{i=1}^k f(t_i)l(I_i)$, where~$l(I_i)=(x_i-x_{i-1})$ for $i=1,2, \dots,k$.

\begin{definition}\label{D:deltafine}

Let $\delta>0$. A partition $p$ is \emph{$\delta$-fine} if every subinterval 
$[x_{i-1}, x_i]$ satisfies $x_i-x_{i-1}<\delta$.

\end{definition}

\begin{definition}\label{D:RieInt}

\textbf{\emph{Riemann Integral.}} A function $f:I \to \mathbb{R}$ is \emph{Riemann integrable} on $I$ if there exists a number $A$ such that for each $\epsilon>0$ there exists a $\delta>0$ such that for all $\delta$-fine tagged partitions $P$ of $I$, we have $|S(f,P)-A|<\epsilon$. We say that $A$ is the Riemann integral of $f$ on $I$ with $A=\int_If$ \emph{\cite{Bar01}}.

\end{definition}
}
{\large
We have now defined the Riemann integral as a limit of Riemann sums associated with a sequence of $\delta$-fine tagged partitions, where the width of each interval $[x_{i-1}, x_i]$ in a given $\delta$-fine partition is less than the constant $\delta$. Yet, it turns out that wherever a function behaves badly, such as the high oscillation of the function in (\ref{denjoy}) as it approaches the origin, these constant $\delta'$s cannot always account for this behavior. Hence many functions are not Riemann integrable, e.g., those with sets of discontinuities greater than measure zero \cite{Sul11}, \cite{Wel11}.

However, with an adjustment to this $\delta$, we can generalize the Riemann integral to handle a larger class of functions than the original Riemann integral or even the Lebesgue integral. The key is to redefine  
$\delta$ as a positive, real-valued function of the point within each interval that has been tagged, which we call the tag. So instead of $\delta$-fine tagged partitions, we have $\delta(x)$-fine tagged partitions \cite{Abb01}, \cite{Bar96}, \cite{Gor94}, \cite{Lee08}, \cite{Swa01}, \cite{Wel11}. In this case, we now pick the tag points first, and based on the values of $\delta(x)$ at these points, we then obtain our intervals within each tagged partition. In this way, we can more carefully refine our intervals within each partition to be arbitrarily narrow wherever a function behaves badly and thus are able to integrate a larger class of functions on $\mathbb{R}$ with the Henstock integral \cite{Abb01}.

So, in looking at our examples, (\ref{dirichlet}) and (\ref{denjoy}), we see that by making the width of the gauge arbitrarily small at the points of discontinuity in the case of (\ref{dirichlet}) 
\begin{equation*}
f(x)= \left \{ 
\begin{array}{lr} 
1 &x\in \mathbb{Q} \\
0 &x\notin \mathbb{Q} \\
\end{array} 
\right.
\end{equation*}

\noindent or at the points of extreme oscillation in the case of  (\ref{denjoy})

\[
f(x)=\int_0^1\frac{1}{x}\sin{\frac{1}{x^3}}dx,
\]

\noindent  we can make the contributions to the Riemann sums arbitrarily small from the subintervals at these badly-behaving points. Therefore, while these two functions are not Riemann integrable \cite{Abb01}, \cite{Sul11}, they are $\epsilon-\delta$ Henstock integrable \cite{Bar01}, \cite{Lee89}, \cite{Sul11}, and  they are also Sequential Henstock integrable due to the equivalence of these two integral definitions, as shown in Theorem \ref{T:e/d.equiv} below. Further, we show directly that (\ref{dirichlet}) is Sequential Henstock integrable in the  example after Definition \ref{D:SeqHen}.

}

{\large
\begin{definition}\label{D:gauge}

A \emph{gauge} is a function $\delta:I \to \mathbb{R}$ such that 
$\delta(x)>0$ for each $x \in I$.

\end{definition}

\begin{definition}\label{D:d(x)fine}
 Let $\delta(x)>0$. A tagged partition $P=\{([x_{i-1}, x_i], t_i)\}_{i=1}^k$ is \emph{$\delta(x)$-fine} if every subinterval 
$[x_{i-1}, x_i]$ satisfies $x_i-x_{i-1}<\delta(t_i)$
\end{definition}

Note: The following lemma provides the foundation for the existence of the Henstock integral defined in Definition \ref{D:e/d} and the Sequential Henstock integral defined in Definition \ref{D:SeqHen}.

\begin{lemma}\label{L:Cousin}

\textbf{\emph{Cousin's Lemma.}} For each gauge $\delta(x)$ defined on $I$, there exists a $\delta(x)$-fine tagged partition of $I$.
\end{lemma}

\noindent The proof is available in \cite{Bar01} and will not be reproduced here.

It is worth noting here that Cousin's Lemma is intricately tied to the completeness of the real line and the fact that the intervals in this case are compact. In fact, Cousin's Lemma can be shown to be equivalent to that of completeness (in conjunction with total boundedness) as well as to the Heine-Borel Theorem in $\mathbb{R}$ \cite{Wel11}. In other words, if for a given set the lemma holds, then the set is also complete, totally bounded, and compact. It is this relationship between the existence of gauges and the real line that provides the foundation for the Henstock integral, as we shall see in the definition below. However, it is also what makes it more difficult to transfer into abstract spaces than, for example, the Lebesgue integral, which readily transfers to, and in many ways is even more naturally defined in, abstract spaces \cite{Sch09}.

\section*{The $\epsilon-\delta$ Henstock Integral \& the Sequential Henstock Integral}

We now define several variations of the Henstock integral and show their equivalence to the Sequential Henstock integral. Definition \ref{D:e/d} is the standard definition for the Henstock integral \cite{Abb01}, \cite{Bar01},  \cite{Swa01}, \cite{Wel11}, and Definition \ref{D:SeqHen} is from Lee  \cite{Lee08}.

\begin{definition}\label{D:e/d}

\textbf{\emph{$\epsilon-\delta$ Henstock Integral.}} A function $f:I \to \mathbb{R}$ is \emph{Henstock integrable} on $I$ if there exists a number $A$ such that for each $\epsilon>0$ there exists a $\delta(x)>0$ such that for all $\delta(x)$-fine tagged partitions $P$ of $I$, we have $|S(f,P)-A|<\epsilon$. We say that $A$ is the Henstock integral of $f$ on $I$ with $A=\int_If$.

\end{definition}

\begin{definition}\label{D:SeqHen}
\emph{\textbf{Sequential Henstock Integral}}. A function $f:I \to \mathbb{R}$ is \emph{Sequential Henstock integrable} on $I$ if there exists a number $A$ and a sequence of positive functions $\{\delta_n(x)\}_{n=1}^\infty$ such that
for every $\delta_n(x)$-fine tagged partition $P_n$, we have  

\begin{equation}\label{E:Seq}
S(f,P_n)=\sum_{i=1}^{m_n \in \mathbb{N}} f(t_{i_n})(x_{i_n}-x_{{(i-1)}_n}) \longrightarrow A \text{ as } n \longrightarrow \infty.
\end{equation}

\noindent We say that $A$ is the Sequential Henstock integral of $f$ on $I$ with $A=\int_If$.

\end{definition}

Note that in both definitions above of the Henstock integral, when $\delta(x)$ is a constant function we simply have the definition of the Riemann integral, thus clearly demonstrating that the Riemann integral is a special case of the Henstock integral, as previously noted \cite{Abb01}, \cite{Bar96}, \cite{Swa01}, \cite{Wel11}. We may also assume that the sequence of gauges is decreasing such that $d_{n+1}(x)<d_n(x)$ for all $x\in I.$

\begin{notation}\label{N:RieSum}
We let $P \ll \delta(x)$ mean that $P$ is a $\delta(x)$-fine tagged partition of an interval \\ \noindent $I=[a,b] \subset  \mathbb{R}$. 
Further, we let $P_n \ll \delta_n(x)$ mean that  for a sequence of tagged partitions $\{P_n\}_{n=1}^\infty$ and a sequence of positive 
$\delta(x)$ functions $\{\delta_n(x)\}_{n=1}^\infty$, each
 $P_n$ is a $\delta_n(x)$-fine tagged partition for $n=1,2,3,\dots$.

\end{notation}

\noindent \textbf{Example: Integrating Dirichlet's Function} \cite{Bar01}, \cite{Lee89}

\noindent We now consider Dirichlet's nowhere-continuous function from (\ref{dirichlet}) and show that it is Sequential Henstock integrable on $I=[0,1]$  although, as noted above, it is not Riemann integrable anywhere on $\mathbb{R}$.

\begin{equation*}\label{dirichlet1}
f(x)= \left \{ 
\begin{array}{lr} 
1  \quad \text{      if       } \quad  x\in [0,1] \text{ is rational} \\
0  \quad \text{      if       }  \quad  x\in [0,1] \text{ is irrational}  
\end{array} 
\right.
\end{equation*}

Let \{$q_m : m\in \mathbb{N}\}$ be an enumeration of the rational numbers in $I=[0,1]$ and let $\epsilon>0.$  Now, let $\{\delta_n(x)\}_{n=1}^\infty$ be a sequence of decreasing (i.e. finer) gauges on $[0,1]$, meaning that $\delta_n(x)<\delta_{n+1}(x)$ for all $x\in I$.  The effect of this condition is that for any tag point $t_i\in [x_{i-1}, x_i]$, the width of each subinterval $I_i=[x_{i-1}, x_i]$ for $i = 1, 2, ...,k$ of any $\delta_n(x)$-fine tagged partition $P_n=\{([x_{i-1}, x_i], t_i)\}_{i=1}^k$ is decreasing for each $\delta_n(x)$ in the sequence. This is because, in practice, we can define $I_i$ in the following manner:
\[
I_i \subseteq [t_i-\tfrac{1}{2}\delta_n(t_i),t_i+\tfrac{1}{2}\delta_n(t_i)] \qquad \text{ for } i = 1, 2, ...,k.
\]
Note that throughout this example, when we write $t_i$ and $I_i$, they belong to a specific tagged partition $P_n$ of $[0,1]$, but for sake of clarity in the notation, we shall refrain from indicating this each time.

We next want to define a gauge in $\{\delta_n(x)\}_{n=1}^\infty$ that will give us a bound on the Riemann sums of $f(x)$ in such a way that all finer gauges will give an even smaller Riemann sum. So, if $t$ is a tag point, we pick $N$ large enough from our sequence of gauges and define

\begin{equation}\label{dirgauge}
 \delta_N(t)= \begin{cases} \epsilon/2^m  &\mbox{if } t = q_m, \\
1  \quad &\mbox{if } t \mbox{ is irrational.}  \end{cases}
\end{equation}

\noindent Now, let $P_n \ll \delta_n(x)$ on $[0,1]$ for $n\geq N$. If the tag $t_i\in I_i$ is irrational, then $f(t_i)=0$ and the contribution of this subinterval (defined using $t_i$) to the Riemann sum is 0.

If $t_i$ is rational, then $f(t_i)=1$, but the contribution to the Riemann sum can be made arbitrarily small since $P_n \ll \delta_n(x)$, meaning that the width of each subinterval $I_i$, which we denote by $l(I_i)=x_i-x_{i-1}$, can be made arbitrarily small for $i = 1, 2, ...,k$. Since the area from each ``strip" will then be $1 \cdot l(I_i)$, we can thus control the area contribution to the Riemann sum of any tag points in $P_n$ that are rational points in $[0,1]$ by defining a subinterval of arbitrarily small width around each one. 

Note that each $P_n$ contains only a finite number of tag points and subintervals by definition, which necessarily means that we do not need to consider every rational or irrational point in $[0,1]$, only those that happen to be in any given $P_n$. There are, however, an infinite number of ways to tag any partition $p$ of any interval $I=[a,b]\subset \mathbb{R}$ \cite{Bar01}.

\bigskip

\noindent More precisely, we have two cases: \\
\indent \emph{Case 1.} If the $m$th rational number $q_m$ is the tag for the interval $I_i$, then 
\[
I_i \subseteq [q_m-\tfrac{1}{2}\delta_n(q_m),q_m+\tfrac{1}{2}\delta_n(q_m)]
\]
so that $l(I_i)\leq \delta_n(q_m)\leq \epsilon/2^m$. 

\emph{Case 2.} If $q_m$ is the tag for any two consecutive (non-overlapping) subintervals $I_i$ and $I_{i+1}$ in $P_n$, meaning that $q_m$ is an endpoint for both, then 
\[
l(I_i)+l(I_{i+1}) \leq \epsilon/2^m
\]
\noindent since in this case 
\[
I_i \subseteq [q_m-\tfrac{1}{2}\delta_n(q_m),q_m] \quad \Longrightarrow \quad l(I_i)\leq \tfrac{1}{2}\delta_n(q_m)
\]
\noindent and
\[
I_{i+1} \subseteq [q_m,q_m+\tfrac{1}{2}\delta_n(q_m)] \quad \Longrightarrow \quad l(I_i)\leq \tfrac{1}{2}\delta_n(q_m).
\]

Therefore, the area contribution to the Riemann sum $S(f,P_n)$ of any ``strip" around a rational point $t_i$ in $f(x)$ is at most $f(t_i) \cdot l(I_i) \leq \epsilon/2^m.$ However, since irrational points make no contribution to $S(f,P_n)$, we can put a bound on the Riemann sum such that
\[
|S(f,P_n)-0| = |S(f,P_n)|\leq \sum_{m=1}^\infty \frac{\epsilon}{2^m} =\epsilon \cdot 1 = \epsilon.
\]
As $\epsilon$ is arbitrary; and  if it is defined, the Sequential Henstock integral of $f(x)$ on $[0,1]$ satisfies $|S(f,P_n)-\int_0^1f|< \epsilon$, then we conclude that Dirichlet's function is Sequential Henstock integrable on $[0,1]$ and that $\int_0^1f=0.$ Further, by Theorem \ref{T:e/d.equiv} that follows, $f(x)$ is also Henstock integrable on $[0,1]$.

\bigskip

\begin{theorem}\label{T:e/d.equiv}

The $\epsilon-\delta$ Henstock integral is equivalent to the Sequential Henstock integral on $I \in \mathbb{R}$.

\end{theorem}

\begin{proof}
We assume that $\{\delta_n(x)\}_{n=1}^\infty$ is a decreasing sequence of gauges such that $\delta_{n+1}(x)<\delta_n(x)$ for all $x\in I$. We shall assume this as needed without stating it in subsequent proofs.

$\Rightarrow$ Suppose that $f:I \to \mathbb{R}$ is integrable in the sense of Definition \ref{D:e/d}. 
Let $\epsilon>0$. There exists a $\delta(x)>0$ such that if $P \ll \delta(x)$, we have 

\begin{equation}\label{E:Rsum}
|S(f,P)-A|<\epsilon.
\end{equation}

For $n=1,2,3,\dots$, let $\epsilon_n$ be a rational $\epsilon$ (as in the definition) such that $0<\epsilon<1$. By Definition \ref{D:e/d}, there exists a $\delta_n(x)$ for each $\epsilon_n$ satisfying (\ref{E:Rsum}). Since $\mathbb{Q}$, 
the set of rational numbers, is countable, $\{\delta_n(x)\}_{n=1}^\infty$ is a sequence. Then, given any $\epsilon>0$,
there exists a $\delta_N(x) \in \{\delta_n(x)\}_{n=1}^\infty$ such that if $n \geq N$ and $P_n$ is any $\delta_n(x)$-fine partition of $I$, then $|S(f,P_n)-A|<\epsilon$. Thus Definition \ref{D:e/d} implies Definition \ref{D:SeqHen}.

$\Leftarrow$ Now, suppose that $f:I \to \mathbb{R}$ is integrable in the sense of Definition \ref{D:SeqHen}. 
There exists a sequence of gauges $\{\delta_n(x)\}_{n=1}^\infty$ such that if $P_n$ is any $\delta_n(x)$-fine partition of $I$, then $|S(f,P_n)-A|<\tfrac{1}{n}$. Let $\lambda>0$. Choose a $\delta_N(x)$ from $\{\delta_n(x)\}_{n=1}^\infty$ such that for a given $\delta(x)>0$, 
\begin{equation}\label{lambda}
|\delta(x)-\delta_N(x)|<\lambda \text{ for all } x \in I. 
\end{equation}
\newline
If we let $\lambda \to 0$, then our choice of $\delta_N(x)$
guarantees that if $P_N \ll \delta_N(x)$, then $P_N \ll \delta(x)$. Hence, for a given $P \ll \delta(x)$, we can make the Riemann sums for $P$ and $P_N$
arbitrarily close (using the same tags in each tagged partition) such that $|S(f,P)-S(f,P_N)|<\tfrac{\epsilon}{2}$. 
 
 Now, let $\epsilon>0$. For a given $\delta(x)>0$, we find a $\delta_N(x)$ satisfying (\ref{lambda}) above and move it down our sequence $\{\delta_n(x)\}_{n=1}^\infty$, denoting its new position in the sequence as $N^*$, so that $1/N^*<\epsilon/2$. Then if  $P \ll \delta(x)$, we have
 \begin{align*}
 |S(f,P)-A| &= |S(f,P)-S(f,P_{N^*})+S(f,P_{N^*})-A| \\
                 &\leq |S(f,P)-S(f,P_{N^*})|+|S(f,P_{N^*})-A| \\
                 &<\tfrac{\epsilon}{2}+\tfrac{1}{N^*}\\
                 &<\tfrac{\epsilon}{2}+\tfrac{\epsilon}{2}=\epsilon.
 \end{align*}
 \noindent Thus Definition \ref{D:SeqHen} implies Definition \ref{D:e/d}, and they are shown to be equivalent.

\end{proof}
\section*{The Darboux Henstock Integral}

\indent We now state the Darboux definition of the Henstock Integral. Definitions \ref{uppersum} and \ref{upperint} are based on Abbott's definitions for the Riemann integral \cite{Abb01}.
\begin{notation}\label{N:uppersum}
Let $H$ denote the set of all possible partitions of an interval $I=[a,b] \subset  \mathbb{R}$, and let $\beta$ denote
the set of all $\delta(x)$-fine partitions of $I$ for a given $\delta(x)>0$. Also, let $P$ denote an arbitrary $\delta(x)$-fine 
partition of $I$ as before.

\end{notation}

\begin{definition}\label{uppersum}

The \emph{lower sum} of $f$ with respect to $P$ is given by
\[
L(f,P)=\inf_{P \subset \beta} \sum_{i=1}^kf(t_i)(x_i-x_{i-1})
\]
Similarly, let the  \emph{upper sum} of $f$ with respect to $P$ be given by

\[
U(f,P)=\sup_{P \subset \beta} \sum_{i=1}^kf(t_i)(x_i-x_{i-1}) \qquad \emph{\cite{Abb01}}.
\]

\end{definition}

\begin{definition}\label{upperint}

The \emph{lower integral} of $f$ is given by
\[
\lowint_If=\sup_{\beta \in H}L(f,P)
\]

\noindent Similarly, let the  \emph{upper integral} of $f$ be given by

\[
\upint_If=\inf_{\beta \in H}U(f,P) \qquad \emph{\cite{Abb01}}.
\]

\end{definition}

\begin{definition}\label{D:Darboux}

\textbf{\emph{Darboux Henstock Integral.}} A function $f:I \to \mathbb{R}$ is \emph{Darboux Henstock integrable} on $I$ if for each $\epsilon>0$ there exists a $\delta(x)>0$ such that for all $\delta(x)$-fine tagged partitions $P$ of $I$, we have $|\upint_If-\lowint_If|<\epsilon$. In this case, we have $\lowint_If=\upint_If=\int_If$ \emph{\cite{Rud76}}.

\end{definition}

\begin{theorem}\label{T:darboux.equiv}

The Darboux Henstock integral is equivalent to the Sequential Henstock integral on $I=[a,b] \subset \mathbb{R}$.

\end{theorem}

\begin{proof}

$\Leftarrow$ Suppose that $f:I \to \mathbb{R}$ is integrable in the sense of Definition \ref{D:SeqHen}. Let $\epsilon>0$. 
 There exists a $\delta_N(x) \in \{\delta_n(x)\}_{n=1}^\infty$ such that if $n \geq N$ and $P_n \ll \delta_n(x)$, then $|S(f,P_n)-\int_If|<\epsilon$. Equivalently, we can say that the sequence $\{S(f,P_n)\} \rightarrow \int_If$ as $n \rightarrow \infty$. Therefore, the upper and lower bounds of the sequence must also converge to $\int_If$, and we have
 $L(f,P_n) \rightarrow \int_If$ and $U(f,P_n) \rightarrow \int_If$ as $n \rightarrow \infty$. 
 
 By the order relationship among the lower and upper bounds and lower and upper integrals of $f$, as $n \rightarrow \infty$ we have 
 \[
 \int_If=L(f,P_n)\leq \sup_{\beta \in H}L(f,P)= \lowint_If\leq\upint_If=\inf_{\beta \in H}U(f,P)\leq U(f,P_n)=\int_If
 \]

Thus, we have $\lowint_If \rightarrow \int_If$ and $\upint_If \rightarrow \int_If$ as $n \rightarrow \infty$ over the set of all $\delta(x)$'s and partitions on $I$. Now, let $\epsilon>0$. There exists a $\delta(x)>0$ such that for all $\delta(x)$-fine tagged partitions $P$ of $I$, we have $|\lowint_If-\int_If|<\tfrac{\epsilon}{2}$ and $|\upint_If-\int_If|<\tfrac{\epsilon}{2}$. This yields
 \begin{align*}
 \left| \upint_If- \lowint_If \right| &= \left| \upint_If- \int_If+ \int_If- \lowint_If \right| \\
          &< \tfrac{\epsilon}{2} + \tfrac{\epsilon}{2}= \epsilon.
\end{align*}
 
\noindent Thus Definition \ref{D:SeqHen} implies Definition \ref{D:Darboux}.

 $\Rightarrow$ Suppose that $f:I \to \mathbb{R}$ is integrable in the sense of Definition \ref{D:Darboux}. Let $\epsilon>0$. There exists a $\delta(x)>0$ such that for all $\delta(x)$-fine tagged partitions $P$ of $I$, we have 
 \[
 \left |\upint_If-\lowint_If \right|<\epsilon
 \]
 so that $\lowint_If=\upint_If=\int_If$. 
 
 Since $\upint_If$ is the greatest lower bound of the upper sums $U(f,P)$, there exists a $\delta_1(x)>0$ such that 
 for all $\delta_1(x)$-fine tagged partitions $P_1$ of $I$, 
 \[
 U(f,P_1)<\upint_If + \tfrac{\epsilon}{2}. 
 \]

 \noindent Similarly for $\lowint_If$, there exists a $\delta_2(x)>0$ such that for all $\delta_2(x)$-fine tagged partitions $P_2$ of $I$, 
 \[
 L(f,P_2)>\lowint_If - \tfrac{\epsilon}{2}. 
 \]

Now, choose $\delta(x)=min\{\delta_1(x),\delta_2(x)\}$ such that for every $\epsilon>0$, there exists a $\delta(x)>0$ such that for $P \ll \delta(x)$, we have
\begin{align}
|U(f,P)-L(f,P)|  &\leq |U(f,P_1)-L(f,P_2)| \label{upperlower} \\
&=  \left|U(f,P_1)- \upint_If +\lowint_If - L(f,P_2)\right| \notag \\
&\leq \left|U(f,P_1)- \upint_If| +|\lowint_If - L(f,P_2)\right|  \notag \\
&< \tfrac{\epsilon}{2} + \tfrac{\epsilon}{2} = \epsilon. \notag
\end{align}

\noindent For the same $P \ll \delta(x)$, we have 
\begin{equation}\label{LSU}
L(f,P)\leq S(f,P) \leq U(f,P), 
\end{equation}

\noindent where $S(f,P)$ is the Riemann sum over $P$. 

\noindent By hypothesis, we have
\[
L(f,P)\leq \sup L(f,P) = \lowint_If = \int_If = \upint_If = \inf U(f,P) \leq U(f,P),
\]
which implies
\begin{equation}\label{UFL}
L(f,P) \leq \int_If \leq U(f,P) \Longrightarrow -U(f,P) \leq -\int_If \leq -L(f,P). 
\end{equation}

\noindent Combining equations (\ref{LSU}) and (\ref{UFL}), and taking $U(f,P)\geq L(f,P)$ by definition,
we have

\begin{equation}\label{SFUL}
\left|S(f,P) -\int_If\right| \leq |U(f,P)-L(f,P)|.
\end{equation}
Now, let $\epsilon>0$. There exists a $\delta(x)>0$ such that for all $\delta(x)$-fine tagged partitions $P$ of $I$, 
we have $|S(f,P) -\int_If|<\epsilon$ by equations (\ref{upperlower}) and (\ref{SFUL}). Thus, we have shown that the
function $f$ is integrable according to Definition \ref{D:e/d}, i.e. $\epsilon-\delta$ Henstock integrable. Theorem
\ref{T:e/d.equiv} asserts that the Sequential Henstock integral is equivalent to the  $\epsilon-\delta$ Henstock 
integral, so we conclude that Definition \ref{D:Darboux} implies Definition \ref{D:SeqHen}; hence they are shown 
to be equivalent.

\end{proof}
\section*{The Topological Henstock Integral}

We now show the equivalence of the Sequential Henstock integral to the Topological Henstock integral defined on a locally compact Hausdorff space. Definitions \ref{D:topg}-\ref{D:TopHen} are from Ahmed and Pfeffer \cite{AP86}.

\begin{notation}\label{setA}
Let $X$ be a locally compact Hausdorff space with a subspace $A \subset X$. We denote the closure
of $A$ as $\bar{A}$ and the interior  as Int $A$. Let $\mathcal{A}$ be a family of subsets of $X$ such that \\
(i) If $A \in \mathcal{A}$, then $\bar{A}$ is compact. \\
(ii) For each $x \in X$, the collection $\mathcal{A}(x)=\{A \in \mathcal{A} \ \vert \ x \in \text{Int }A \}$ is a neighborhood
base at $x$. \\
(iii) If $A, B \in  \mathcal{A}$, then $A \cap B \in \mathcal{A}$, and there exist disjoint sets $C_1, \dots , C_k$ 
in $\mathcal{A}$ such that $A-B=\cup_{i=1}^kC_i$.
\end{notation}

\begin{definition}\label{D:topg}

A \emph{gauge} (topological) on $A \in \mathcal{A}$ is a map $U$ assigning to each $x \in \bar{A}$
a neighborhood $U(x)$ of $x$ contained in $X$ \emph{\cite{AP86}}.

\end{definition}

\begin{definition}\label{D:division}

A \emph{division} (topological) of $A \in \mathcal{A}$ is a disjoint collection $\{A_1, \dots ,A_k \} \subset 
\mathcal{A}$ such that $\cup_{i=1}^kA_i=A$ \emph{\cite{AP86}}.

\end{definition}

\begin{definition}\label{D:toppart}

A \emph{partition} (topological) of $A \in \mathcal{A}$ is a set $P=\{(A_1,t_1), \dots ,(A_k,t_k)\}$ such that
$\{A_1, \dots ,A_k \}$ is a division of $A$ and $\{x_1, \dots ,x_k \} \subset \bar{A}$. If $U$ is a gauge on $A$, we 
say the partition $P$ is $U$-fine if $A_i \subset U(x_i)$, for $i=1,2, \dots ,k$ \emph{\cite{AP86}}.

\end{definition}

\begin{definition}\label{D:vol}

 A \emph{volume} (topological)   is a non-negative function such that $V(A)=\sum_{i=1}^kv(A_i)$ 
 for every $A \in \mathcal{A}$ and each division $\{A_1, \dots ,A_k \}$ of $A$ \emph{\cite{AP86}}.
 \end{definition}
 
 \noindent Note: Volume here can 
 intuitively be defined to represent the ``length" of the``interval" $A$.

\begin{definition}\label{D:TopHen}

\textbf{\emph{Topological Henstock Integral.}} Let $X$ be a locally compact Hausdorff space, and let $A \in \mathcal{A}$
with $f:\bar{A} \to \mathbb{R}$. Then $f$ is Topological Henstock integrable if for each $\epsilon>0$, there exists
a $U(x)$ such that 
\[
\left| \sum_{i=1}^k f(t_i)v(A_i)-\int_Af \right| = \left| \sigma(f,P)-\int_Af \right| <\epsilon
\]
for every $U(x)$-fine partition $P$ of $A$ \emph{\cite{AP86}}.

\end{definition}

\begin{theorem}\label{T:top.equiv}

The Topological Henstock integral is equivalent to the Sequential Henstock integral on $I=[a,b] \subset \mathbb{R}$.

\end{theorem}

By the Heine-Borel Theorem, each closed interval $[a,b] \subset \mathbb{R}$ is compact \cite{Rud76}. Hence, any point $x \in \mathbb{R}$ is contained in an open interval $(a,b)$, which in turn is contained in the compact subspace $[a,b]$, so that $\mathbb{R}$ is a locally compact Hausdorff space \cite{Mun00}. Therefore, we prove the result
for $X=\mathbb{R}$ and $\mathcal{A}=\{[a,b]\ \vert \ a,b \in \mathbb{R}, \ a<b \}$. However, we show that under these
conditions, Definition \ref{D:TopHen}, the Topological Henstock integral, reduces to Definition \ref{D:e/d}, the $\epsilon-\delta$ Henstock integral, which is equivalent to the Sequential Henstock integral in Theorem \ref{T:e/d.equiv}. 

\begin{proof}

For $A \in \mathcal{A}$, let $v(A_i)=x_i-x_{i-1}$ and define a gauge $U_\delta(t_i)$ on $A$ such that 
$U_\delta(t_i)=(t_i-\delta(t_i), t_i+\delta(t_i)), \ i=1,2, \dots, k$, is a $U$-fine partition $P$ on $A$ with $\delta(x)>0$ for 
each $x \in [a,b]$. Thus, in this case, we can say that $P$ is $\delta(x)$-fine since each subinterval $[x_i, x_{i-1}]$ will
be contained in each open interval $(t_i-\delta(t_i), t_i+\delta(t_i)), \ i=1,2, \dots, k$. 
Therefore, for closed intervals $[a,b]$ in $\mathbb{R}$, Definition \ref{D:TopHen} becomes the following:

Let $\epsilon>0$. There exists a $\delta(x)>0$ such that for all $\delta(x)$-fine tagged partitions $P$, we have 
\[
\left| \sum_{i=1}^k f(t_i)(x_i-x_{i-1}) - \int_a^bf \right| = \left| S(f,P)- \int_a^bf \right|<\epsilon,
\]
which is precisely the $\epsilon-\delta$ Henstock integral definition. Thus, by Theorem \ref{T:e/d.equiv}, we conclude 
that Definition \ref{D:TopHen}, the Topological Henstock integral, is equivalent to Definition \ref{D:SeqHen}, the Sequential
Henstock integral, when the sets in $\mathcal{A}$ are $A=[a,b] \subset \mathbb{R}$.

\end{proof}

}

\section*{Basic Properties of the Sequential Henstock Integral}
{\large
We now establish the basic properties and fundamental theorems of the Sequential Henstock integral. The statements in the following theorems and lemma have been adapted according to the Sequential Henstock integral. The statements of the theorems and lemma for the standard $\epsilon-\delta$ definition of the Henstock integral can be found in  \cite{Bar01},  \cite{Swa01}, and \cite{Wel11}.

\begin{theorem}\label{T:unique}

If a function $f:I \to \mathbb{R}$ is Sequential Henstock integrable on $I=[a,b] \subset \mathbb{R}$, then the value of the 
integral on $I$ is unique \emph{\cite{Swa01}}.

\end{theorem}

\begin{proof}

Let $\int_If=A$ and $\int_If=B$. We show $A=B$. Let $\epsilon>0$. According to Definition \ref{D:SeqHen}, 
there exists a $\delta_M(x) \in \{\delta_m(x)\}_{m=1}^\infty$ such that if $m \geq M\in \mathbb{N}$ and $P_m\ll 
\delta_m(x)$,
 then  
 \[
|S(f,P_m)-A|<\tfrac{\epsilon}{2}.
 \]
 
\noindent Further, there exists a $\delta_K(x) \in \{\delta_k(x)\}_{k=1}^\infty$ such that if $k \geq K\in \mathbb{N}$ and $P_k\ll \delta_k(x)$,
 then  
 \[
|S(f,P_k)-A|<\tfrac{\epsilon}{2}.
 \]

Now, for each $\delta_m(x)\in \{\delta_m(x)\}_{m=1}^\infty$ and each $\delta_k(x)\in \{\delta_k(x)\}_{k=1}^\infty$, choose \\
$\delta_n(x)=$min$\{\delta_m(x),\delta_k(x)\}$ for $n, m, k =1,2,3, \dots$ and for all $x\in I$. Clearly, $\delta_n(x)$ is a sequence of gauges on $I$. Let $P_n\ll \delta_n(x)$. Then $P_n\ll \delta_m(x)$ and $P_n\ll \delta_k(x)$. 

\noindent For $\epsilon>0$, there exists a $\delta_N(x) \in \{\delta_n(x)\}_{n=1}^\infty$ such that if $n \geq N$ and 
$P_n\ll \delta_n(x)$, then  
\begin{align*}
|A-B| &=|A-S(f,P_n)+S(f,P_n)-B| \\
&\leq |A-S(f,P_n)|+|S(f,P_n)-B| \\
&<\tfrac{\epsilon}{2}+\tfrac{\epsilon}{2}=\epsilon.
\end{align*}


\end{proof}

\begin{theorem}\label{T:algprop}

Let $f,g:I\to \mathbb{R}$ be Sequential Henstock integrable functions on $I=[a,b] \subset \mathbb{R}$, and let $c \in \mathbb{R}$ \emph{\cite{Swa01}}.

\indent 1. If $f\geq0$ on $I$, then $\int_If \geq0$.

\indent 2. $cf$ is Sequential Henstock integrable on $I$, and $\int_I{cf}=c\int_If.$

\indent 3. $f+g$ is Sequential Henstock integrable on $I$, and $\int_I{(f+g)}=\int_If+\int_Ig.$

\indent 4. If $f(x)\leq g(x)$ for all $x \in I$, then $\int_If\leq \int_Ig.$

\indent 5. If $|f|$ is Sequential Henstock integrable on $I$, then $\left| \int_If \right|\leq \int_I|f|.$

\indent 6. Let $M \in \mathbb{R}$. If $|f(x)|\leq M$ for all $x \in I=[a,b]$, then  $\left| \int_a^bf \right|\leq M(b-a).$

\end{theorem}

\begin{proof}
(1.) Let $f\geq0$ on $I$ and $f\in H^*(I)$ (the class of Sequential Henstock integrable functions). Then for every
$\epsilon>0$, there exists a $\delta_N(x) \in \{\delta_n(x)\}_{n=1}^\infty$ such that if $n \geq N$ and $P_n\ll \delta_n(x)$,
 then $|S(f,P_n)-\int_If|\leq \epsilon.$ Equivalently, we have
 
\begin{equation}\label{part5a}
 \left| \sum_{i=1}^kf(t_i)l_n(I_i)-\int_If \right|\leq \epsilon,
 \end{equation}
 
 \noindent where $l_n(I_i)$ is the length of the $i^{th}$ subinterval in the $n^{th}$ $\delta_n(x)$-fine partition $P_n$. 
 From (\ref{part5a}), we have
\begin{equation}\label{part5a2}
 -\epsilon \leq \sum_{i=1}^kf(t_i)l_n(I_i)-\int_If\leq \epsilon.
 \end{equation}
 $f(t_i)\geq 0$ for every $t_i\in P_n$ and $l_n(I_i)>0$ by definition, so  $0\leq \sum f(t_i)l_n(I_i)=S(f,P_n)$. 
 Combining this with (\ref{part5a2}), we have
 \[
0\leq S(f,P_n)\leq \int_If+\epsilon \text{\qquad for all } n\geq N.
\]
Since $\epsilon$ is arbitrary, we have  $\int_If \geq0$.

(2.) Let $f\in H^*(I)$, $k\in \mathbb{R},$ and
$\epsilon>0$. There exists a $\delta_N(x) \in \{\delta_n(x)\}_{n=1}^\infty$ such that $n \geq N$ and $P_n\ll \delta_n(x)$
 implies $\left| S(f,P_n)-\int_If \right| \leq \tfrac{\epsilon}{|c|}.$ Then, for $n\geq N$ we have
 \begin{align*}
 \left| S(cf,P_n)-c\int_I{f} \right| &=\left| \sum_{i=1}^kcf(t_i)l_n(I_i)-c\int_If \right| \\
  &= \left| c \left( \sum_{i=1}^kf(t_i)l_n(I_i)-\int_If \right) \right| \\
  &= |c| \cdot \left| \sum_{i=1}^kf(t_i)l_n(I_i)  - \int_If \right| \\
   &=|c| \cdot \left| S(f,P_n)-\int_If \right| \\
 &<|c| \cdot \tfrac{\epsilon}{|c|} = \epsilon.
 \end{align*}
 
 \noindent Thus, $cf\in H^*(I)$, and $\int_I{cf}=c\int_If.$
 
 (3.) Let $\epsilon>0$. Since $f\in H^*(I)$ there exists a $\delta_M(x) \in \{\delta_m(x)\}_{m=1}^\infty$ such that 
 \\ $m \geq M\in \mathbb{N}$ and $P_m\ll \delta_m(x)$
 implies $\left| S(f,P_m-\int_If \right| \leq \tfrac{\epsilon}{2}.$
 Similarly, since $g\in H^*(I)$,  \\ there exists a $\delta_K(x) 
 \in \{\delta_k(x)\}_{k=1}^\infty$ such that 
 $k \geq K\in \mathbb{N}$ and $P_k\ll \delta_k(x)$
 implies \\ $\left| S(g,P_k)-\int_If \right| \leq \tfrac{\epsilon}{2}.$
 
 Now for each $\delta_m(x)\in \{\delta_m(x)\}_{m=1}^\infty$ and each $\delta_k(x)\in \{\delta_k(x)\}_{k=1}^\infty$, choose \\
$\delta_n(x)=$min$\{\delta_m(x),\delta_k(x)\}$ for $n, m, k =1,2,3, \dots$ and for all $x\in I$. Clearly, $\delta_n(x)$ is a sequence of gauges on $I$. Let $P_n\ll \delta_n(x)$. Then $P_n\ll \delta_m(x)$ and $P_n\ll \delta_k(x)$. 

For every $\epsilon>0$, there exists a $\delta_N(x) \in \{\delta_n(x)\}_{n=1}^\infty$ such that 
  $n \geq N$ and $P_n\ll \delta_n(x)$ 
 implies 
 \begin{align*}
 \left| S(f+g,P_n)- \left( \int_If+\int_Ig \right) \right| &=  \left| \sum_{i=1}^p(f+g)(t_i)l_n(I_i)- \left( \int_If+\int_Ig \right) \right| \\
 &= \left| \sum_{i=1}^pf(t_i)l_n(I_i)- \int_If + \sum_{i=1}^pg(t_i)l_n(I_i) - \int_Ig  \right| \\
&\leq \left| \sum_{i=1}^pf(t_i)l_n(I_i)- \int_If \right| + \left| \sum_{i=1}^pg(t_i)l_n(I_i) - \int_Ig  \right| \\
&< \tfrac{\epsilon}{2} + \tfrac{\epsilon}{2}=\epsilon.
 \end{align*}
 Thus $f+g\in H^*(I)$, and $\int_I{f+g}=\int_If+\int_Ig.$
 
 (4.) Suppose $f,g \in H^*(I)$ and $f(x)\leq g(x)$ for all $x \in I$. Let $h=g-f\geq 0.$ By Part (3) above, $h=g-f \in H^*(I)$ and $\int_I(g-f)=\int_Ig
 +\int_I(-f)=\int_Ig -\int_If$, by Part (2). Further, by Part (1), $\int_Ih\geq0.$ Thus
 \[
 0\leq \int_Ih = \int_I(g-f) = \int_Ig -\int_If \quad \Longrightarrow \quad \int_If \leq \int_Ig.
 \]
 
 (5.) Let $f, |f| \in H^*(I)$. Since $-|f| \leq f \leq |f|,$ then by Part (4) above, we have
 \[
 -\int_I|f|\leq \int_If\leq \int_I|f|.
 \]
 Therefore, $\left| \int_If \right|\leq \int_I|f|.$
 
 (6.) Let $f \in H^*(I)$ and $M \in \mathbb{R}$ such that $|f(x)|\leq M$ for all $x \in I=[a,b]$. Then $-M\leq f \leq M,$ 
 so by Part (4) above, we have
 \[
  -\int_a^bM \leq \int_a^bf\leq \int_a^bM \quad \Longrightarrow \quad \left| \int_a^bf \right| \leq \int_a^bM=M(b-a)
 \]

\end{proof}

\begin{center}
\textbf{Additivity \& Cauchy Criterion for Sequential Henstock Integrability}
\end{center}

We now define the right-left procedure \cite{Bar01} for picking tag points in a tagged partition, which allows us to have some or all of the tags be endpoints of subintervals.

\begin{definition}\label{D:rightleft}

Let $P=\{([x_{k-1}, x_k], t_k)\}_{k=1}^m$ be a tagged partition of an interval $I=[a,b]$, and let $t_k$ be an interior point of the 
subinterval $[x_{k-1}, x_k].$ Using the \emph{right-left procedure}, we then let $P^*$ be formed from $P$ by adding a new partition point $\eta=t_k$ so that
\[
a=x_0<x_1< \dots <x_{k-1}< \eta < x_k < \dots <x_{m-1}<x_m=b.
\]
Now tag both subintervals $[x_{k-1}, \eta]$ and $[\eta, x_k]$ with the tag $\eta=t_k$ so that $\eta$ is the right endpoint of the first of these subintervals and the left endpoint of the second. 

\noindent Since
\[
f(t_k)(x_k-x_{k-1})=f(t_k)(\eta-x_{k-1})+f(t_k)(x_k-\eta),
\]
we have $S(f,P)=S(f,P^*)$ \emph{\cite{Bar01}}.

\end{definition}
 
 We note that this process can be reversed  in order to join two abutting subintervals that share the tag $t_k$ so that the new subinterval has $t_k$ as a tag on the interior. Therefore, the right-left procedure allows us to assume that 

 (i) all of the tags are endpoints of their respective subintervals, \\
\indent (ii) no tag, except possibly $a$ or $b$, is an endpoint of the subintervals, or \\
\indent (iii) no point is the tag of two distinct subintervals.

\begin{theorem}\label{T:additivity}

\emph{\textbf{Additivity.}} Let $f:[a,b]\to \mathbb{R}$ and $c\in (a,b)$. If $f$ is a Sequential Henstock integrable function on $[a,c]$ and $[c,b]$, then $f$ is Sequential Henstock integrable on $[a,b]$ and 
\[
\int_a^bf=\int_a^cf+\int_c^bf \qquad \emph{\cite{Swa01}}.
\]

\end{theorem}

\begin{proof}

Let $\epsilon>0$. Since $f\in H^*([a,c])$ there exists a $\delta_M(x) \in \{\delta_m(x)\}_{m=1}^\infty$ such that 
 \\ $m \geq M\in \mathbb{N}$ and $P_m\ll \delta_m(x)$
 implies $\left| S(f,P_m)-\int_a^cf \right| \leq \tfrac{\epsilon}{2}.$
 Similarly, since $f\in H^*([c,b])$,  \\ there exists a $\delta_K(x) 
 \in \{\delta_k(x)\}_{k=1}^\infty$ such that 
 $k \geq K\in \mathbb{N}$ and $P_k\ll \delta_k(x)$
 implies \\ $\left| S(f,P_k)-\int_c^bf \right| \leq \tfrac{\epsilon}{2}.$
 
 We define a gauge \cite{Bar01} $\delta_N(x) \in \{\delta_n(x)\}_{n=1}^\infty$ in order to force the point $c$ to be a tag of each 
$P_n\ll \delta_n(x)$ for $n\geq N$. Using the right-left procedure in Defintion \ref{D:rightleft}, we then split apart each partition $P_n$ at the 
tag $c$ so that it becomes a partition point of each $P_n$.
\begin{equation}\label{addgauge}
\text{Let } \delta_N(t)= \left \{ 
\begin{array}{lr} 
min \{\delta_m(t),\tfrac{1}{2}(c-t)\} &t\in[a,c) \\
min \{\delta_m(c), \delta_k(c)\} &t=c \\
min\{\delta_k(t),\tfrac{1}{2}(t-c)\} &t\in(c,b]
\end{array} 
\right.
\qquad  \text{for } m\geq M \text{ and } k\geq K.
\end{equation}

Now, let $P_n\ll \delta_n(x)$ for $n\geq N$. Let $P_m$ be the sequence of partitions of $[a,c]$ consisting of the partition points $P_n \cap [a,c]$ and  $P_k$  the sequence of partitions of $[c,b]$ consisting of the partition points $P_n \cap [c,b]$ for $n,m,k=1,2,3,\dots.$ 

\noindent Then the right-left procedure provides that 
\[
S(f,P_n)=S(f,P_m)+S(f,P_k).
\]
Given $\epsilon>0$, there exists a $\delta_N(x) \in \{\delta_n(x)\}_{n=1}^\infty$ such that for $n\geq N$ we have

\begin{align*} 
\left| S(f,P_n) -\left( \int_a^cf + \int_c^bf \right) \right| &= \left| \sum_{i=1}^jf(t_i)l_n(I_i) -\left( \int_a^cf + \int_c^bf \right) \right| \\
&=\left| \sum_{i=1}^pf(t_i)l_m(I_i)+ \sum_{i=p}^jf(t_i)l_k(I_i) -\left( \int_a^cf + \int_c^bf \right) \right| \\
& =\left| S(f,P_m) - \int_a^cf + S(f,P_k) -\int_c^bf \right| \\
&\leq  \left| S(f,P_m) - \int_a^cf \right| + \left| S(f,P_k) -\int_c^bf \right| \\
& <\tfrac{\epsilon}{2} + \tfrac{\epsilon}{2}=\epsilon.
\end{align*}
Therefore, $f\in H^*([a,b])$ and $\int_a^bf=\int_a^cf+\int_c^bf.$
\end{proof}

We now prove the Cauchy Criterion Sequential Henstock integrability, which will allow us, among other things, to prove the converse of Theorem \ref{T:additivity}; i.e., if $f\in H^*(I)$, then $f$ is Sequential Henstock integrable on any compact subinterval of $I$.

\begin{theorem}\label{T:cauchy}

\emph{\textbf{Cauchy Criterion for Sequential Henstock Integrability.}} Let $f:I \to \mathbb{R}$. Then $f$ is Sequential Henstock integrable on $I$ if and only if for every $\epsilon>0$, there exists a $\delta_N(x) \in \{\delta_n(x)\}_{n=1}^\infty$ such that for $n\geq N$ and for all $\delta_n(x)$-fine tagged partitions $P_n$ and $Q_n$ of $I$ we have
\[
|S(f,P_n)-S(f,Q_n)|<\epsilon \qquad \emph{\cite{Bar01}}.
\]

\end{theorem}

\begin{proof}
$\Rightarrow$ Suppose $f\in H^*(I)$. Let $\epsilon>0$. There exists a $\delta_N(x) \in \{\delta_n(x)\}_{n=1}^\infty$ such that for $n\geq N$ and for $P_n \ll \delta_n(x)$, we have $|S(f,P_n)-\int_If|<\tfrac{\epsilon}{2}.$ Now, if $P_n \ll \delta_n(x)$ and $Q_n \ll \delta_n(x)$, then
\begin{align*}
|S(f,P_n)- S(f,Q_n)| &=  \left| S(f,P_n) -\int_If+\int_If -S(f,Q_n)  \right| \\
&\leq \left| S(f,P_n) -\int_If \right|+ \left| \int_If -S(f,Q_n)  \right| \\
&< \tfrac{\epsilon}{2} + \tfrac{\epsilon}{2}=\epsilon
\end{align*}

$\Leftarrow$  Let $\epsilon>0$. Suppose there exists a $\delta_N(x) \in \{\delta_n(x)\}_{n=1}^\infty$ such that for $n\geq N$, \\$P_n \ll \delta_n(x)$, and $Q_n \ll \delta_n(x)$, we have $|S(f,P_n)-S(f,Q_n)|<\tfrac{1}{n}.$ We now construct a Cauchy Sequence of Riemann sums \cite{Wel11}, which converges (by the Cauchy Criterion for $\mathbb{R}$) to a real number we shall denote by $S^*$. Note that without loss of generality, we may assume that $\{\delta_n(x)\}_{n=1}^\infty$ is a decreasing sequence for all $x \in I$. Thus, for any $m>n$, $P_m$ is $\delta_n(x)$-fine, and letting $n \to \infty,$ we see that $|S(f,P_n)-S(f,Q_n)|<\tfrac{1}{n}$ is a Cauchy Sequence. Hence,  
\[
\{S(f,P_n)\}_{n=1}^\infty \longrightarrow  S^* \quad \text{ as } n \to \infty,
\]
and we write that for every $\epsilon>0$, there exists an $n\in \mathbb{N}$ such that for $n\geq N$, we have $|S(f,P_n)-S^*|<\tfrac{1}{n}.$

Now, let $\epsilon>0$. There exists a $\delta_N(x) \in \{\delta_n(x)\}_{n=1}^\infty$ such that for $n\geq N$, where $\tfrac{1}{N}<\tfrac{\epsilon}{2}$, and for $P_n \ll \delta_n(x)$ and $Q_n \ll \delta_n(x)$, we have
\begin{align*}
|S(f,Q_n)-S^*| &= |S(f,Q_n)-S(f,P_n)+S(f,P_n)-S^*| \\
&\leq |S(f,Q_n)-S(f,P_n)|+|S(f,P_n)-S^*| \\
&< \tfrac{1}{N}+\tfrac{1}{N}<\tfrac{\epsilon}{2} + \tfrac{\epsilon}{2}=\epsilon.
\end{align*}
Thus $f\in H^*(I)$ and $\int_If=S^*$.

\end{proof}

Using the Cauchy Criterion for Sequential Henstock integrability, we now prove the converse of Theorem \ref{T:additivity}; i.e., if $f$ is Sequential Henstock integrable on $I$, then $f$ is Sequential Henstock integrable on any compact subinterval of $I$, which we denote as $I_i$.

\begin{theorem}\label{T:decompos}

Let $f:I \to \mathbb{R}$. If $f\in H^*(I)$, then $f\in H^*(I_i)$ for $i=1,2,3,\dots,k$ where $k\in \mathbb{N}$ \emph{\cite{Wel11}}.
\end{theorem}

\begin{proof}
We first prove the case for $k=2$ and then show the general result by induction. Let $\epsilon>0$. There exists a $\delta_N(x) \in \{\delta_n(x)\}_{n=1}^\infty$ such that for $n\geq N$ and for $P_n \ll \delta_n(x)$, we have $|S(f,P_n)-\int_If|<\tfrac{\epsilon}{2}.$ Let $I_1$ and $I_2$ be intervals such that $I_1 \cap I_2 \neq \emptyset$ and $I_1 \cup I_2=I$ \cite{Wel11}. We denote the restrictions of $\delta_n(x)$ to  
 $I_1$ and $I_2$ as $\delta_{1_n}(x)$ and $\delta_{2_n}(x)$, respectively. Then, for $P_{1_n} \ll \delta_{1_n}(x)$ and $P_{2_n} \ll \delta_{2_n}(x)$, we have 
\[
S(f,P_n)=\sum_{i=1}^2f(t_i)l_n(I_i)=f(t_1)l_n(I_1)
+f(t_2)l_n(I_2)=S(f,P_{1_n})+S(f,P_{2_n}).
\]
We observe that  if $Q_{1_n} \ll \delta_{1_n}(x)$ and $P_{2_n} \ll \delta_{2_n}(x)$, then $R_n=Q_{1_n}\cup P_{2_n}$ is $\delta_n(x)$-fine for \\ $n=1,2,3,\dots.$ Note that both of the previous results hold, by induction, for any finite number of subintervals of $I$, which we shall use below. Now, let $\epsilon>0$. There exists a $\delta_N(x) \in \{\delta_n(x)\}_{n=1}^\infty$ such that for $n\geq N$ and $P_{1_n}, Q_{1_n} \ll \delta_{1_n}(x)$ on $I_1$, we have
\begin{align*}
|S(f,P_{1_n})-S(f,Q_{1_n})| &= \left|S(f,P_{1_n})+S(f,P_{2_n})-\int_If-S(f,Q_{1_n})-S(f,P_{2_n})+\int_If \right| \\
&\leq  \left|S(f,P_{1_n})+S(f,P_{2_n})-\int_If\right| + \left| S(f,Q_{1_n})+S(f,P_{2_n})-\int_If \right| \\
&= \left|S(f,P_n)-\int_If\right| +\left| S(f,R_n)+\int_If \right| \\
&<\tfrac{\epsilon}{2} + \tfrac{\epsilon}{2}=\epsilon.
\end{align*}

A similar result follows for $I_2$. This proves the assertion for the $k=2$ case. Now, assume this holds for the $k^{\text{th}}$ case. We show that it holds for the $(k+1)^{\text{th}}$ case. Let $\epsilon>0$. There exists a $\delta_N(x) \in \{\delta_n(x)\}_{n=1}^\infty$ such that for $n\geq N$ and $P_{{(k+1)}_n}, Q_{{(k+1)}_n} \ll \delta_{{(k+1)}_n}(x)$ on $I_{k+1}$, we have
\begin{align*}
|S(f,P_{{(k+1)}_n})-S(f,Q_{{(k+1)}_n})| &\leq \left|S(f,P_{{(k+1)}_n})+\left[ S(f,P_{1_n})+S(f,P_{2_n})+\dots 
+S(f,P_{k_n}) \right] -\int_If \right| \\
&+ \left|S(f,Q_{{(k+1)}_n})+\left[ S(f,P_{1_n})+S(f,P_{2_n})+\dots 
+S(f,P_{k_n}) \right] -\int_If \right| \\
&= \left|S(f,P_n)-\int_If\right| +\left| S(f,R_n)+\int_If \right| \\
&<\tfrac{\epsilon}{2} + \tfrac{\epsilon}{2}=\epsilon.
\end{align*}
Thus, $f\in H^*(I_i)$ for all $i=k\in \mathbb{N}.$

\end{proof}

}
\section*{Fundamental Theorems of the Sequential Henstock Integral}
{\large
We now establish the Fundamental Theorem of Calculus, Parts I and II, using the Sequential Henstock integral as 
well as Integration by Parts and Substitution theorems. We then prove Uniform Convergence and Monotone 
Convergence for the Sequential Henstock integral, from which Dominated Convergence follows.

\begin{theorem}\label{T:FTC1}
\emph{\textbf{Fundamental Theorem of Calculus Part I.}} If $F:I\to \mathbb{R}$ is differentiable at every point in $I$, then $f$ is Sequential Henstock integrable and 
$\int_a^bf=F(b)-F(a)$ \emph{\cite{Abb01}}.
\end{theorem}

\begin{proof}
Let $f=F'$ and $\epsilon>0$. Since $F$ is differentiable for every $x\in I$, there exists a $\delta_N(x) \in \{\delta_n(x)\}_{n=1}^\infty$ such that for each $x,t\in I$ and for $n\geq N$, if $0<|x-t|<\delta_n(t),$ we have
\[
\left| \frac{F(x)-F(t)}{x-t}-F'(t) \right| <\epsilon
\]
Then
\[
|F(x)-F(t)-F'(t)(x-t) | <\frac{\epsilon|x-t|}{b-a}.
\]
Further, $P_n \ll \delta_n(x)$ for all  $x,t\in I$ implies
\begin{align}
|S(f,P_n) - [F(b)-F(a)]| &= \left|\sum_{i=1}^kf(t_i)l_n(I_i)- [F(b)-F(a)]\right| \notag \\
&= \left|\sum_{i=1}^kf(t_i)(x_i-x_{i-1})_{\delta_n}- \sum_{i=1}^k[F(x_i)-F(x_{i-1})] \right| \notag \\
&= \left|\sum_{i=1}^k(f(t_i)(x_i-x_{i-1})_{\delta_n}- [F(x_i)-F(x_{i-1})]) \right| \notag \\
&\leq \sum_{i=1}^k|f(t_i)(x_i-x_{i-1})_{\delta_n}- [F(x_i)-F(x_{i-1})]| \label{E:FTC1}
\end{align}
Using the right-left procedure as a basis to assume that no tag $t_i$ appears as an endpoint of any subinterval, we split
apart each subinterval at its tag \cite{Wel11} so that $[x_{i-1},x_i]$ becomes $[x_{i-1},t_i] \cup [t_i,x_i]$. We also employ the 
telescoping sum to allow us to estimate the resulting inequalities. Now,  (\ref{E:FTC1}) becomes
\begin{align*}
|S(f,P_n) - [F(b)-F(a)]| &\leq \sum_{i=1}^k|f(t_i)(x_i-x_{i-1})_{\delta_n}- [F(x_i)-F(x_{i-1})]|  \\
&= \sum_{i=1}^k|f(t_i)(x_i-t_i)_{\delta_n}- [F(x_i)-F(t_i)] | \\
&+ \sum_{i=1}^k|f(t_i)(t_i-x_{i-1})_{\delta_n}- [F(t_i)-F(x_{i-1})]| \\
&<\frac{\epsilon \sum_{i=1}^k |x_i-t_i|}{b-a} + \frac{\epsilon \sum_{i=1}^k |t_i-x_{i-1}|}{b-a} = \frac{\epsilon \sum_{i=1}^k |x_i-t_i+t_i-x_{i-1}|}{b-a} \\
&=\frac{\epsilon \sum_{i=1}^k |x_i-x_{i-1}|}{b-a}=\frac{\epsilon (b-a)}{b-a}=\epsilon.
\end{align*}
Thus, $f\in H^*(I)$ and $\int_a^bf=F(b)-F(a)$.

\end{proof}

\begin{theorem}\label{T:FTC2}
\emph{\textbf{Fundamental Theorem of Calculus Part II.}} Let $f\in H^*([a,b])$ and let $F(x)=\int_a^xf$ for each $x\in [a,b]$. 
If $f$ is continuous on $[a,b]$, then $F$ is differentiable at $x\in [a,b]$ with $F'=f$ \emph{\cite{Swa01}}.
\end{theorem}

\begin{proof}
Let $\epsilon>0$. Since $F$ is continuous on $[a,b]$, there exists a $\delta_N(x) \in \{\delta_n(x)\}_{n=1}^\infty$ such that for each $x,t\in [a,b]$ and for $n\geq N$, if $0<|x-t|<\delta_n(t),$ we have $|f(x)-f(t)| <\epsilon$. This yields 
\begin{equation}\label{E:FTC2}
f(x)-\epsilon<f(t)<f(x)+\epsilon.
\end{equation}

Now, let $0<|y-x|<\delta_n(t)$ such that $a<x\leq t \leq y<b$ and integrate over $[x,y]$ \cite{Swa01}. From Theorem \ref{T:decompos}, we know that  $f\in H^*([x,y])$. Further, from Theorem \ref{T:additivity} and the fact that $F(x)=\int_a^xf$, we observe that
\[
\int_x^yf=\int_a^yf-\int_a^xf=F(y)-F(x)
\]
Thus by Theorem \ref{T:algprop}.4 and (\ref{E:FTC2}), we have 
\begin{align}
 & \int_x^y(f(x)-  \epsilon)dt <  \int_x^yf(t)dt< \int_x^y(f(x)+ \epsilon)dt  \notag \\
 &\Rightarrow (f(x)-\epsilon) (y-x)< \int_x^yf(t)dt< (f(x)+\epsilon)(y-x)  \notag  \\
&\Rightarrow  f(x)-\epsilon < \frac{\int_x^yf(t)}{y-x}< f(x)+\epsilon  \notag  \\
&\Rightarrow  \left| \frac{\int_x^yf(t)}{y-x}-f(x) \right| \notag \\
&= \left| \frac{F(y)-F(x)}{y-x}-f(x) \right| <\epsilon. \label{E:FTC2.2}
\end{align}

\noindent A similar argument shows that (\ref{E:FTC2.2}) holds for $y< x$. Therefore, $F'(x)=f(x)$ and $F$ is differentiable 
on $[a,b]$ since $x$ is arbitrary in $[a,b]$.
\end{proof}

\begin{theorem}\label{T:sub}

\emph{\textbf{Substitution.}} Let $f_1, f_2$ be differentiable on $[a,b]$. Then $f'_1f_2\in H^*([a,b])$ if and only if 
$f_1f'_2\in H^*([a,b])$. Then 
we have $\int_a^bf'_1f_2 =f_1(b)f_2(b)-f_1(a)f_2(a)- \int_a^bf_1f'_2$ \emph{\cite{Swa01}}.
\end{theorem}

\begin{proof}
From the product rule, we have $(f_1f_2)'=f'_1f_2+f_1f'_2=\alpha$. Using Theorem \ref{T:algprop}.3, the Fundamental Theorem of Calculus Part I, and the 
fact that $f_1f_2$ is an anti-derivative for $\alpha$, we observe that
\begin{align}
 \int_a^bf'_1f_2+\int_a^bf_1f'_2  \  \Leftrightarrow \ \int_a^b(f'_1f_2+f_1f'_2)   \   \Leftrightarrow  \  \int_a^b(f_1f_2)'   \  \Leftrightarrow  \   f_1(b)f_2(b)-f_1(a)f_2(a)  \label{sub1} 
\end{align}
Now, equating the far left and right sides of (\ref{sub1}) yields
\[
 \int_a^bf'_1f_2 \ =  \ f_1(b)f_2(b)-f_1(a)f_2(a)  -\int_a^bf_1f'_2
\]
Therefore, $f'_1f_2\in H^*([a,b])$ if and only if 
$f_1f'_2\in H^*([a,b])$.

\end{proof}

\begin{theorem}\label{T:parts}

\emph{\textbf{Integration by Parts.}} Let $f, \psi$ be differentiable on $[a,b]$. Then 
\[
\int_{\psi(a)}^{\psi(b)}f'=\int_a^b(f' \circ \psi)\psi' \qquad \emph{\cite{Swa01}}.
\]
\end{theorem}

\begin{proof}
Using the chain rule to find an anti-derivative, we have $(f \circ \psi)' =(f' \circ \psi)\psi'.$ Now, we apply the Fundamental Theorem of Calculus twice and use the transitive property of $\mathbb{R}.$
\begin{align}
\int_{\psi(a)}^{\psi(b)}f' &= f \circ \psi(b) - f \circ \psi(a) \label{parts1} \\
 \int_a^b(f' \circ \psi)\psi' = \int_a^b(f \circ \psi)'&= f \circ \psi(b) - f \circ \psi(a) \label{parts2} 
\end{align}
By (\ref{parts1}) and (\ref{parts2}), we have $\int_{\psi(a)}^{\psi(b)}f'=\int_a^b(f' \circ \psi)\psi'.$

\end{proof}
\bigskip
\begin{center}
\textbf{Uniform Convergence Theorem}
\end{center}

Next, we prove the Uniform Convergence Theorem for the Sequential Henstock integral, which is a weak but useful result \cite{Bar01}. The proof follows Bartle's \cite{Bar01}.

\begin{theorem}\label{uniform}
\emph{\textbf{Uniform Convergence.}} If $\{f_k\} \in H^*(I)$ converges to $f$ uniformly on $I$, then $f\in H^*(I)$ and 
$$\int_If=\lim_{k \to \infty} \int_If_k \qquad \emph{\cite{Bar01}}.$$

\end{theorem}

\begin{proof}

 Let  $\epsilon>0$. Since $f_k \to f$ uniformly, there exists a $K\in \mathbb{N}$ such that $k\geq K$ implies $|f_k-f|<\epsilon$ for all points in $I$. Thus, we have
\begin{align}
|f_h-f_k| &= |f_h-f+f-f_k| \notag \\
&\leq |f_h-f|+|f-f_k| \notag \\
&< \epsilon+\epsilon=2\epsilon \ \qquad  \qquad  \qquad  \text{for } h,k\geq K \text{ and } x\in I. \label{unf1}
\end{align}

\noindent Let $I=[a,b]$. From Theorem \ref{T:algprop}.4,  \ref{T:algprop}.3, and (\ref{unf1}), we observe that $\text{for } h,k\geq K 
\text{ and } x\in I$
\begin{align}
&-2\epsilon < f_h-f_k < 2\epsilon  \notag \\
&  \Rightarrow \ -\int_a^b 2\epsilon < \int_a^b (f_h-f_k) < \int_a^b 2\epsilon \notag \\
& \Rightarrow -2\epsilon(b-a) < \int_a^b f_h-\int_a^bf_k < 2\epsilon(b-a)  \notag \\
&  \Rightarrow \left| \int_a^b f_h-\int_a^bf_k \right| < 2\epsilon(b-a) \label{unf2}.
\end{align}

\noindent Hence, $\{\int_If_k\}_{k=1}^\infty$ is a Cauchy Sequence in $\mathbb{R}$, which converges to a real number we shall denote by $A$; equivalently, $\lim_{k \to \infty} \int_If_k=A$. We now show that $f\in H^*(I)$ and $\int_If=A$.

Let $\epsilon>0$. For any sequence of tagged partitions $\{P_n\}$, there exists a $K\in \mathbb{N}$ such that $k\geq K$ implies 
\begin{align*}
|S(f_k,P_n)-S(f,P_n)| &= \left |\sum_{i=1}^pf_k(t_i)l_n(I_i) - \sum_{i=1}^pf(t_i)l_n(I_i)\right| \\
& \leq \sum_{i=1}^p |[f_k(t_i)-f(t_i)]l_n(I_i)|  \\
& < \epsilon \sum_{i=1}^p  l_n(I_i) \qquad \text{(since $f_k \to f$ uniformly)} \\
&=\epsilon(b-a) \qquad \ \ \text{(by telescoping sum of the endpoints of $I_i$.)}
\end{align*}

Now, fix $\gamma \geq K$ such that $\int_If_\gamma \to A$ so that we have 
$\left|\int_If_\gamma - A \right| <\epsilon$. Since $f_\gamma \in H^*(I)$, for $\epsilon>0$, there exists a 
$\delta_{\gamma_N}(x) \in \{\delta_{\gamma_n}(x)\}_{n=1}^\infty$ such that for $n\geq N$ and for all $\delta_{\gamma_n}(x)$-fine tagged partitions $P_n$, we have
\[
\left| S(f_\gamma,P_n)-\int_If_\gamma \right| <\epsilon.
\]
Consequently,
\begin{align*}
|S(f,P_n)-A| &= \left |S(f,P_n)-S(f_\gamma,P_n)+S(f_\gamma,P_n)-\int_If_\gamma+\int_If_\gamma -A\right| \\
& \leq \left |S(f,P_n)-S(f_\gamma,P_n) \right|+\left|S(f_\gamma,P_n)-\int_If_\gamma \right| +\left| \int_If_\gamma -A\right| \\
&<\epsilon(b-a) + \epsilon +\epsilon = \epsilon(b-a+2).
\end{align*}

\noindent Since $\epsilon$ is arbitrary, we conclude that $f\in H^*(I)$ and 
$\int_If=\lim_{k \to \infty} \int_If_k.$
\end{proof}

We observe that uniform convergence implies convergence of the sequence of Riemann sums of $f_k$ to the Riemann sum 
of the limit function $f$. In fact, a stronger theorem is possible if the convergence of $f_k \to f$ is only pointwise. Then, the limit can be interchanged such that $\int_If=\lim_{k \to \infty} \int_If_k$ if and only if $S(f_k,P_n)\to S(f,P_n)$ as $k\to \infty$. See \cite{Bar01}, pg. 127-128, for a proof of this interesting result.

\begin{center}
\textbf{Henstock's Lemma \& the Monotone Convergence Theorem}
\end{center}

We now establish Henstock's Lemma for the Sequential Henstock integral in order to prove the Monotone Convergence Theorem for the Sequential Henstock integral. We recall that a monotone function on $\mathbb{R}$ is one that is either non-decreasing or non-increasing over its entire domain.

\begin{lemma}\label{henstock}
\emph{\textbf{Henstock's Lemma.}} Let $f:[a,b]\to \mathbb{R}$ be a Sequential Henstock integrable function. If
$P_n^*=\{([x_{i-1}, x_i], t_i)\}_{i=1}^p$ is any $\delta_n(x)$-fine set of disjoint, tagged subintervals of $[a,b]$, then
\begin{equation}\label{hen1}
\left| S(f,P^*_n)-\sum_{i=1}^p \int_{x_{i-1}}^{x_i}f \right| \leq \epsilon.
\end{equation}

\noindent Further,
\begin{equation}\label{hen2}
\sum_{i=1}^p \left| f(t_i)(x_i-x_{i-1})- \int_{x_{i-1}}^{x_i}f \right| \leq 2\epsilon \qquad \emph{\cite{Wel11}}.
\end{equation}
\end{lemma}

\begin{proof}
Let $\epsilon>0$. Since $f\in H^*(I)$, there exists a $\delta_N(x) \in \{\delta_n(x)\}_{n=1}^\infty$ such that for $n\geq N$ and for $P_n \ll \delta_n(x)$, we have 
\begin{equation}\label{hen3}
\left |S(f,P_n)-\int_If \right|\leq \epsilon.
\end{equation}

 We shall utilize the subintervals of $I=[a,b]$ that do not overlap any of the subintervals contained in $P^*_n$. We then form tagged partitions of these subintervals not in $P^*_n$ that are fine enough to force their Riemann sums to be arbitrarily small \cite{Bar01}. If $P^*_n = P_n$, then there is nothing to prove. Otherwise, let $\{I_j \ \vert 1\leq j \leq m \}$ be a collection of non-overlapping, closed subintervals of $[a,b]$, and let $S$ denote the set $\cup_{i=1}^p[x_{i-1}, x_i]$ of closed subintervals in $P^*_n$ such that $\{I_j\} \cup S$ is a partition of $[a,b]$. 

Let $\eta>0$. By Theorem \ref{T:decompos}, we know that $f$ is integrable on each subinterval $I_j$, for $j=1,2,\dots,m$. Thus, 
there exists a $\delta_{j_N}(x) \in \{\delta_{j_n}(x)\}_{n=1}^\infty$ such that for $n\geq N$ and for all $\delta_{j_n}(x)$-fine tagged partitions $P_{j_n}$ of $I_j$, we have
 \begin{equation}\label{hen4}
\left |S(f,P_{j_n})-\int_{I_j}f \right|\leq \frac{\eta}{m}.
\end{equation}

We may assume that $\delta_{j_n}(x)\leq \delta_n(x)$ for all $x\in I_j$. Let $P_n$  denote the tagged partition of $I$ such that
 $P_n=P^*_n + \cup_{i=1}^mP_{j_n}$ for all $n\geq N$. We observe that $P_n$ is $\delta_n(x)$-fine and that (\ref{hen3}) holds. In addition, if $k=p+m\in \mathbb{N}$, we have
 \[
 \sum_{i=1}^k |f(t_i)l_n(I_i)| = \sum_{i=1}^p |f(t_i)l_n(I_i)| + \sum_{j=1}^m \sum_{i=1}^{r_j} |f(t_{j_i})l_{j_n}(I_{j_i})| 
 \]
which yields
 \begin{equation}\label{hen5}
 S(f,P_n)= S(f,P^*_n)+ \sum_{j=1}^mS(f,P_{j_n})
\end{equation}
 
 Now, since $S$ and $\{I_j\}$ are collections of closed subintervals of $I=[a,b]$, by Theorem \ref{T:additivity} and \ref{T:decompos}, we have
 \begin{equation}\label{hen6}
 \int_If= \int_Sf+\sum_{j=1}^m \int_{I_j}f.
\end{equation}
 
 \noindent Putting together (\ref{hen3}), (\ref{hen4}), (\ref{hen5}), and (\ref{hen6}), we have
 \begin{align*}
\left| S(f,P^*_n)- \int_Sf \right| &= \left| S(f,P^*_n)-\sum_{i=1}^p \int_{x_{i-1}}^{x_i}f \right| \\
&= \left| S(f,P_n)-\sum_{i=1}^m S(f,P_{j_n}) - \left(  \int_If -\sum_{i=1}^m \int_{I_j}f \right) \right| \\
&\leq \left| S(f,P_n)- \int_If \right| + \left|\sum_{i=1}^m S(f,P_{j_n})-\sum_{i=1}^m \int_{I_j}f  \right| \\
&\leq \left| S(f,P_n)- \int_If \right| + \sum_{i=1}^m \left|S(f,P_{j_n})- \int_{I_j}f  \right| \\
&\leq \epsilon+m\left(\frac{\eta}{m}\right)=\epsilon +\eta.
\end{align*}
Since $\eta$ is arbitrary, \ref{hen1} is proved.

Now, to prove \ref{hen2}, we separate the $p$ tagged subintervals in $P^*_n$ into two $\delta_n(x)$-fine subpartitions $\sigma_1$  and $\sigma_2$ \cite{Bar01}. Let $\sigma_1$ represent the collection of tagged subintervals such that
\[
f(t_i)l_n(I_i)-\int_{I_i}f=f(t_i)(x_i-x_{i-1})\geq0, 
\]
and let $\sigma_2$
represent the collection of tagged subintervals such that
\[
f(t_i)l_n(I_i)-\int_{I_i}f=f(t_i)(x_i-x_{i-1})<0.
\]

\noindent Applying Part 1 of the lemma, (\ref{hen1}), to both $\sigma_1$  and $\sigma_2$, we have
\begin{align}\label{hen2.1}
 \sum_{i \in\sigma_1}\left |f(t_i)l_n(I_i)-\int_{I_i}f\right| &=  \sum_{i \in\sigma_1}\left (f(t_i)l_n(I_i)-\int_{I_i}f\right) \notag \\ &=
   \left| \sum_{i \in\sigma_1} \left( f(t_i)l_n(I_i)-\int_{I_i}f \right) \right|  \leq \epsilon
\end{align}
and
\begin{align}\label{hen2.2}
 \sum_{i \in\sigma_2}\left |f(t_i)l_n(I_i)-\int_{I_i}f\right| &= - \sum_{i \in\sigma_2}\left (f(t_i)l_n(I_i)-\int_{I_i}f\right)\notag \\ &=
  \left| \sum_{i \in\sigma_2} \left( f(t_i)l_n(I_i)-\int_{I_i}f \right) \right|   \leq \epsilon.
\end{align}
 
 \noindent Adding the inequalities in (\ref{hen2.1}) and (\ref{hen2.2}), we observe that
 \[
 \sum_{i=1}^p \left| f(t_i)(x_i-x_{i-1})- \int_{x_{i-1}}^{x_i}f \right| =
 \sum_{i=1}^p \left |f(t_i)l_n(I_i)-\int_{I_i}f\right|\leq 2\epsilon
 \]
 Thus,  (\ref{hen2}) is proved.

\end{proof}

\begin{theorem}\label{mono}

\emph{\textbf{Monotone Convergence.}} Let $\{f_k\}$ be a monotone sequence of Sequential Henstock integrable functions defined on $I=[a,b]$ that converges pointwise to a limit function $f:[a,b]\to \mathbb{R}$. If $\lim_{k \to \infty} \int_If_k$ exists, then $f\in H^*(I)$ and $$\int_If=\lim_{k \to \infty} \int_If_k \qquad \emph{\cite{Bar01}}.$$
\end{theorem}

\begin{proof}
We consider the case of a non-decreasing sequence of functions. A similar result follows for a non-increasing sequence. Let $\epsilon>0$. Since $\lim_{k \to \infty} \int_If_k$ exists, there exists a $K\in \mathbb{N}$ such that $k\geq K$ implies  $|\int_If_k-A| <\tfrac{\epsilon}{3}.$ Further, since $f_k\in H^*(I)$, for each $k\in \mathbb{N}$ there exists a $\delta_{k_N}(x) \in \{\delta_{k_n}(x)\}_{n=1}^\infty$ such that for $n\geq N$ and  $P_n \ll \delta_{k_n}(x)$ on $I$, we have
 \begin{equation}\label{mono1}
\left |S(f_k,P_n)-\int_If_k \right|<\frac{\epsilon}{3\cdot 2^{k+1}}.
\end{equation}

Lastly, since $\{f_k\} \to f$ pointwise on $I$, then for each $x\in I$, there exists a positive integer $K$ such that if $k(x)\geq K$, we have
 \begin{equation}\label{mono2}
 |f_{k(x)}-f|<\frac{\epsilon}{3(b-a)}.
\end{equation}
\newline
Now, let $\{\delta_n(x)\}$ be a sequence of gauges on $I$ and let $\{P_n\}$ be a 
sequence of $\delta_n(x)$-fine partitions of $I$. We note here that while a $\delta_N(x)$ is chosen to match a particular $\epsilon$ in the definition of the Sequential Henstock integral, in this case each given $\epsilon$ necessitates a possibly different term $k$ in the sequence of functions $\{f_k\}$ in order to satisfy (\ref{mono2}) since $k$ is dependent upon the point $x\in I$ that we use. (For this reason, we write $k(x)$ instead of $k$.) This is due to the rather weak condition of pointwise convergence on $\{f_k\} \to f$. Therefore, the choice of $\delta_N(x)$ to match a given $\epsilon$ will necessarily depend on the $k(t_i)$ for $i=1,2,...,m$, where $t_i$ is a tag point of $P_n$, but for the sake of clarity, we avoid notating this each time.

 The primary part of our strategy is to use this sequence of gauges to separate the set of tagged subintervals in any given $P_n$ into classes of different widths \cite{Bar01}, \cite{Wel11} based on the corresponding $\delta_n(t_i)$, which depends on $k(t_i)$; i.e. which $f_{k(t_i)}$ in our sequence of functions we have to use for a given point $t_i$ to satisfy the pointwise convergence condition. Therefore, for a given $P_n$,  we use the finite set $\{k(t_i)\}$ for $i=1,2,...,m$ to distinguish our classes. We then use Henstock's Lemma to show that the Riemann sums of each class of subintervals for a given
$P_n$ converge to the Sequential Henstock integral taken over each class of subintervals. We shall also employ the convergence of $\{\int_I f_k\}$ and the pointwise convergence $f_k \to f$, as assumed by hypothesis. 

Since $A$ is the value of $\lim_{k=1}^\infty \int_If_k$, we want to show that $\int_If \to A$. We  take the definition of the Sequential Henstock integral of $f$ on $I$ and use the triangle inequality to establish a three-term inequality, determining a bound for each term separately. Thus, we have
 \begin{align}\label{mono3}
\left |S(f,P_n)-A \right|&\leq  \left| \sum_{i=1}^m f(t_i)l_n(I_i) - \sum_{i=1}^m f_{k(t_i)}(t_i)l_n(I_i) \right| \notag \\ &+ 
\left| \sum_{i=1}^m f_{k(t_i)}(t_i)l_n(I_i) - \sum_{i=1}^m \int_{I_i}f_{k(t_i)} \right| \notag \\ &+ \left| \sum_{i=1}^m \int_{I_i}f_{k(t_i)}-A \right|.
\end{align}

\noindent The 1st term on the right side in (\ref{mono3}) is bounded by
\begin{align}
\left| \sum_{i=1}^m f(t_i)l_n(I_i) - \sum_{i=1}^m f_{k(t_i)}(t_i)l_n(I_i) \right| 
 &= \left| \sum_{i=1}^m \left[ f(t_i)l_n(I_i) -  f_{k(t_i)}(t_i)l_n(I_i) \right] \right| \notag \\
&\leq \sum_{i=1}^m \left|  f(t_i)-f_{k(t_i)}(t_i) \right| l_n(I_i)\notag\\
&\leq \frac{\epsilon(b-a)}{3(b-a)}=\frac{\epsilon}{3} \qquad \qquad \text{(by telescoping sum)}. \label{mono3.5}
\end{align}

\noindent  The 2nd term on the right side in (\ref{mono3}) is bounded by the triangle inequality by
 \begin{equation}\label{mono4}
\sum_{i=1}^m \left| f_{k(t_i)}(t_i)l_n(I_i) -  \int_{I_i}f_{k(t_i)} \right|
\end{equation}

\noindent We use Bartle's approach \cite{Bar01} to estimate the sum in (\ref{mono4}). 

 Let  $s=max\{k(t_i), \dots, k(t_m)\}$, with each member of this set $\geq K$ by (\ref{mono2}). Since each partition $P_n$ in the sequence of $\delta_n(x)$-fine partitions $\{P_n\}$ has a finite number of tags, this maximum is well defined.  Because the positive integer $k(t_i)$ is by no means unique within the classes of tagged subintervals in any given $P_n\in \{P_n\}$, we take the sum first over all values of $i$ such that $k(t_i)=p\in \mathbb{N}$ for $p\geq K$; and then over $p=K, \dots,s$. 

Now, fix $p$ and consider all tags $t_i$ such that $k(t_i)=p$. There are at most $m$ of them. Here we notate that the choice of $\delta_n(t_i)$ depends on $p$. By definition,
$$I_i \subseteq [t_i-\tfrac{1}{2}\delta_{p_n}(t_i), t_i+\tfrac{1}{2}\delta_{p_n}(t_i)] \qquad \text{for } i\leq m$$ 

\noindent for each $\delta_{p_n}(x)\in \{\delta_{p_n}(x)\}_{n=1}^\infty$. Therefore these tagged subintervals $I_i$ form a \\ $\delta_{p_n}(x)$-fine subpartition of $I$. Hence, with (\ref{mono1}) satisfying the hypothesis for Henstock's Lemma for each $f_k$, we have 
 \begin{equation}\label{mono5.5}
\sum_{k(t_i)=p} \left| f_{k(t_i)}(t_i)l_n(I_i) -  \int_{I_i}f_{k(t_i)} \right| <\frac{2\epsilon}{3\cdot 2^{p+1}}=\frac{\epsilon}{3\cdot 2^p}.
\end{equation}

\noindent Next, sum over $p=K, \dots,s$, which yields 
 \begin{equation}\label{mono5}
\sum_{p=K}^s \frac{\epsilon}{3\cdot 2^p} = \frac{\epsilon}{3}\sum_{p=K}^s \frac{1}{2^p}<
 \frac{\epsilon}{3}\sum_{p=1}^\infty \frac{1}{2^p}= \frac{\epsilon}{3}.
\end{equation}

For the third term on the right side of (\ref{mono3}), we observe that $\{f_k\}$ increasing implies  if $K\leq k(t_i) \leq s$, then
$f_K\leq f_{k(t_i)} \leq f_s$. Theorem \ref{T:algprop}.4 implies that $\int_{I_i}f_K\leq \int_{I_i}f_{k(t_i)} \leq \int_{I_i}f_s$. Summing over $I_i$
for $i=1,2,\dots,m$, we have 
\[
\sum_{i=1}^m \int_{I_i}f_K\leq \sum_{i=1}^m \int_{I_i}f_{k(t_i)} \leq \sum_{i=1}^m \int_{I_i}f_s,
\]
which by Theorem \ref{T:additivity} becomes
\begin{align}
\int_{I}f_K\leq \sum_{i=1}^m \int_{I_i}f_{k(t_i)} \leq  \int_{I}f_s. \label{mono6}
\end{align}

\noindent But for $k(t_i)\geq K$, we observe that
\begin{align}
\left |\int_If_s-A\right|<\frac{\epsilon}{3}
\Rightarrow \int_If_s < A +\frac{\epsilon}{3}, \label{mono7}
\end{align}
and
\begin{align}
\left |\int_If_K-A\right|<\frac{\epsilon}{3}
\Rightarrow -\frac{\epsilon}{3}<\int_If_K -A \Rightarrow A-\frac{\epsilon}{3}<\int_If_K. \label{mono8}
\end{align}

\noindent Combining (\ref{mono6}), (\ref{mono7}), and (\ref{mono8}), we have
\begin{align}
A-\frac{\epsilon}{3} < \int_{I}f_K\leq \sum_{i=1}^m \int_{I_i}f_{k(t_i)} \leq  \int_{I}f_s< A+\frac{\epsilon}{3}, 
\end{align}

\noindent which implies
 \begin{align}
\left| \sum_{i=1}^m \int_{I_i}f_{k(t_i)} -A\right|< \frac{\epsilon}{3}. \label{mono9}
\end{align}

\noindent Using the bounds from (\ref{mono3.5}), (\ref{mono5}), and (\ref{mono9}) in the inequality from (\ref{mono3}), we have
\begin{align*}\label{mono3}
\left |S(f,P_n)-A \right|&\leq  \left| \sum_{i=1}^m f(t_i)l_n(I_i) - \sum_{i=1}^m f_{k(t_i)}(t_i)l_n(I_i) \right|  \\ 
&+ \left| \sum_{i=1}^m f_{k(t_i)}(t_i)l_n(I_i) - \sum_{i=1}^m \int_{I_i}f_{k(t_i)} \right|  \\ 
&+ \left| \sum_{i=1}^m \int_{I_i}f_{k(t_i)}-A \right| \\
&< \frac{\epsilon}{3}+\frac{\epsilon}{3}+\frac{\epsilon}{3}=\epsilon \qquad \text{for } P_n \ll \delta_n(x).
\end{align*}

\indent Thus $f\in H^*(I)$ and $\int_If=\lim_{k \to \infty} \int_If_k.$

\end{proof}

Note: Using Theorem \ref{mono} and Fatou's Lemma, the Dominated Convergence Theorem for Sequential Henstock integrable functions follows readily \cite{Bar01}. Therefore, we refer  the interested reader to Bartle \cite{Bar01}, pg. 121-123, as well as to Lee and Vyborny \cite{LV00} for a discussion of another important convergence theorem for the Henstock integral, that of Controlled Convergence. 
}

{\large

}
\section*{Summary}
{\large

In this paper, we discussed the limitations of the Riemann integral as well as those of the Lebesgue integral and pointed to the motivation for researchers who have sought a more general yet relatively simple integral on $\mathbb{R}$, such as finding an integral that can handle nowhere-continuous functions like (\ref{dirichlet}) and extreme oscillation functions like (\ref{denjoy}). To this end, Henstock and Kurzweil independently developed the Henstock integral in the mid-1950s, which when defined as a limit of Riemann sums, is only a slight modification of the Riemann integral. This they achieved by using a positive function, or gauge, that varies the width of each subinterval in any tagged partition instead of using a constant width for each subinterval.

While the standard definition of the Henstock integral uses an $\epsilon-\delta$ definition, we developed the Sequential Henstock integral by first showing its equivalence to the $\epsilon-\delta$ definition as well as to the Darboux definition using the limits of upper and lower sums and to a topological definition taken on a locally compact Hausdorff space---in this case $\mathbb{R}$. Further, we proved the basic properties and fundamental theorems of the Sequential Henstock integral, such as the Cauchy Criterion for Sequential Henstock Integrability, the Fundamental Theorem of Calculus, and the Monotone Convergence Theorem, thereby establishing a rigorous foundation for the Sequential Henstock integral upon which to perform additional research.

}
\section*{Conclusions \& Suggestions for Further Study}
{\large

After a thorough review of the literature on the Henstock integral, it is clear that interest in the topic peaked in the 1990s and early 2000s. The reason we found for this is that further applications of the integral were not uncovered and the Lebesgue integral is more readily generalized to abstract spaces than the Henstock integral is (see the remarks following Lemma \ref{L:Cousin}) although abstraction has been developed to some extent \cite{Sch09}.

Up until this research, however, the theory of Henstock integration did not include definitions and theorems based on sequences, and it is our opinion that the Sequential Henstock integral can be used to renew the interest of integration theorists and researchers in the Henstock integral.  In line with the stated goals, the results of this research can now be extended to investigations in abstract spaces and the applications that arise from this as well as to the inclusion of the Sequential Henstock integral in introductory calculus courses to assess the possible pedagogical benefits. Several suggestions for further study are as follows:


\indent (1.) Develop the Sequential Henstock integral theorems for various classes of functions, such as step functions, regulated functions, measurable functions, absolutely integrable functions, absolutely continuous functions, and generalized absolutely continuous functions (see \cite{Bar01} for definitions). Using these, directly show the equivalence between the Sequential Henstock integral and the Lebesgue integral, Denjoy integral, Perron integral, McShane integral, and others (see \cite{Gor94} for definitions). 

\indent (2.) Establish the Sequential Henstock integral in more abstract and general settings, starting first with infinite intervals in $\mathbb{R}$ and then extending to $\mathbb{R}^n$, metric spaces, or topological vector spaces. For example, using the generalized notion of a sequence in topology, the \emph{net} \cite{Mun00}, it may be possible to sufficiently generalize the Sequential Henstock integral to define it in topological spaces, such as metric spaces, normed vector spaces, or function spaces like Banach and Hilbert spaces (see \cite{Bar01} for definitions). 
}

{\large

}
\end{document}